\documentclass[12pt]{amsart}

\usepackage{cancel}
\usepackage[normalem]{ulem}

\usepackage{amsthm,amsmath,amstext,amsfonts,amssymb}
\usepackage{fullpage}
\usepackage{url}
\usepackage{txfonts}
\usepackage{mathrsfs}
\usepackage[osf]{mathpazo}
\usepackage{xcolor}
\usepackage{graphicx}
  \usepackage[all]{xy}
 
\usepackage{marginnote}
\usepackage[pdfauthor={Renaud Coulangeon and Achill Schürmann},
            pdftitle={Energy minimization, periodic sets and spherical designs - Part II}]{hyperref}

\newcommand{\Z}{\mathbb{Z}}
\newcommand{\R}{\mathbb{R}}

\newcommand{\C}{\mathbb{C}}

\DeclareMathOperator{\Tr}{\operatorname{Tr}}
\DeclareMathOperator{\GL}{\operatorname{GL}}
\DeclareMathOperator{\SL}{\operatorname{SL}}
 
\newcommand{\Aut}{\operatorname{Aut}}
\newcommand{\Isom}{\operatorname{Isom}}

\newcommand{\Sn}{{\mathcal S}^{n}}
\DeclareMathOperator{\id}{Id}

\DeclareMathOperator{\grad}{grad}
\DeclareMathOperator{\hess}{hess}

\def\norm#1{\left\Vert#1\right\Vert}
\def\set#1{\left\{#1\right\}}
\def\br#1{\left[#1\right]}
\def\pr#1{\left(#1\right)}

\newcommand{\eg}{{\it e.g. }}
\newcommand{\ie}{{\it i.e. }}

\newcommand{\lc}{{\it loc. cit. }}

\renewcommand{\Im}{\operatorname{Im}}

\theoremstyle{plain}

\newtheorem{theorem}{Theorem}[section]
\newtheorem{proposition}[theorem]{Proposition}

\newtheorem{lemma}[theorem]{Lemma}

\theoremstyle{definition}
\newtheorem{defn}[theorem]{Definition} 

\theoremstyle{definition}
\newtheorem*{defn*}{Definition} 

\theoremstyle{remark}
\newtheorem{rem}[theorem]{Remark}

\theoremstyle{remark}

\theoremstyle{remark}

\theoremstyle{remark}
\newtheorem*{ex*}{Example}
\theoremstyle{remark}
\newtheorem*{exs*}{Examples}

\begin{document}

\title{Local Energy Optimality of Periodic Sets}

\author[R.~Coulangeon and A.~Sch\"urmann]{Renaud Coulangeon and Achill Sch\"urmann}





\address{%
Universit\'e de Bordeaux, 
Institut de Math\'ematiques, 
351, cours de la 
\linebreak\indent 
Lib\'e\-ration, 
33405 Talence cedex, France} 
\email{Renaud.Coulangeon@math.u-bordeaux1.fr}

\address{%
Institute for Mathematics, 
University of Rostock,
18051 Rostock,
Germany}
\email{achill.schuermann@uni-rostock.de}

\begin{abstract}
We study the local optimality of periodic point sets in $\R^n$ for 
energy minimization in the Gaussian core model, that is, 
for radial pair potential functions $f_c(r)=e^{-c r}$ with $c>0$.
By considering suitable parameter spaces for $m$-periodic sets, 
we can locally rigorously analyze the energy of point sets, within the family of 
periodic sets having the same point density.
We derive a characterization of periodic point sets being
$f_c$-critical for all $c$ in terms of weighted spherical $2$-designs
contained in the set.
Especially for $2$-periodic sets like the family $\mathsf{D}^+_n$
we obtain expressions for the hessian of the energy function,
allowing to certify $f_c$-optimality in certain cases.
For odd integers $n\geq 9$ we can hereby in particular show that $\mathsf{D}^+_n$
is locally $f_c$-optimal among periodic sets for all sufficiently
large~$c$.
\end{abstract}

\maketitle

\setcounter{tocdepth}{1}
\tableofcontents

\section{Introduction}

Point configurations which minimize energy for a given pair                     
potential function occur in diverse branches of mathematics and its applications.
There are various numerical approaches to find locally stable configurations.
However, in general, proving optimality of a point configuration
appears hardly possible, except maybe for some very special sets.

In \cite{MR2257398} Cohn and Kumar introduced the notion of a {\em universally
optimal point configuration}, that is, a set of points in a given space,
which minimizes energy for all completely monotonic potential functions.
There exist several fascinating examples among spherical point sets.
However, considering infinite point sets in Euclidean spaces is more
difficult. Even a proper definition of potential energy bears subtle convergence problems.
For \textit{periodic sets} such problems can be avoided,
so that these point configurations are the ones usually considered in
the Euclidean setting. 
When working with local variations of periodic sets it is convenient to
work with a parameter space up to translations and orthogonal
transformations, as introduced in \cite{MR2466406}.
With it, a larger experimental study of energy minima 
among periodic sets in low dimensions ($n\leq 9$) 
was undertaken 
in the {\em Gaussian core model}, that is, 
for potential functions $f_c(x) := e^{ - c \, x}$, with $c > 0$
(see \cite{PhysRev09}). 
These experiments support a conjecture of Cohn and Kumar that the 
hexagonal lattice $\mathsf{A}_2$ in dimension~$2$ and the root lattice $\mathsf{E}_8$ 
in dimension~$8$ are universally optimal among periodic sets in their dimension.
Somewhat surprising, the numerical experiments also suggest that the 
root lattice $\mathsf{D}_4$ in dimension $4$ is universally optimal.
Since proving global optimality seemed out of reach, 
we considered a kind of local universal optimality among periodic sets
in \cite{MR2889159}.
We showed that lattices whose shells are spherical $4$-designs and which
are locally optimal among lattices can not locally be improved to another
periodic set with lower energy.
By a result due to Sarnak and Str\"ombergsson \cite{MR2221138}, this implies
local universal optimality among periodic sets 
for the lattices $\mathsf{A}_2$, $\mathsf{D}_4$ and $\mathsf{E}_8$, as well as for the 
exceptional Leech lattice $\Lambda_{24}$. 
A corresponding result for the ``sphere packing case'' $c\to \infty$
is shown in~\cite{MR3074813}.

In all other dimensions the situation is much less clear.
In dimension $3$, for instance, there is a small intervall
for $c$ with a {\em phase transition}, for which periodic point-configurations 
seem not to minimize energy at all.
For all larger~$c$ the fcc-lattice (also known as $\mathsf{D}_3$) and for all smaller~$c$ the
bcc-lattice (also known as $\mathsf{D}^\ast_3$) appear to be energy minimizers. 
Similarly, there appear to be no universal optima in dimensions $5$, $6$ and~$7$. 
Contrary to a conjecture of Torquato and
Stillinger from 2008~\cite{PhysRev08}, there even
seem to be various non-lattice configurations which  
minimize energies in each of these dimensions.
Quite surprising, the situation appears to be very different in dimension~$9$:
According to our numerical experiments
it is possible that there exists a universally optimal $2$-periodic
(non-lattice) set in dimension~$9$. This set, known as $\mathsf{D}^+_9$,
is a union of two translates of the root lattice $\mathsf{D}_9$.
From the viewpoint of energy minimization, respectively our numerical experiments, 
$\mathsf{D}^+_9$ seems almost of a similar nature as the
exceptional lattice structures $\mathsf{E}_8$ and $\Lambda_{24}$.
However, its shells are only spherical $3$-designs (and not
$4$-designs), which makes a major difference for our proofs.
The purpose of this paper is to shed more light onto 
the energy minimizing properties of $\mathsf{D}^+_9$
and similar periodic non-lattice sets that might exist in other dimensions.
Here, we in particular derive criteria for $f_c$-critical 
periodic point sets (Theorem~\ref{th:periodic-critical}) and we show that
$\mathsf{D}^+_9$ is locally $f_c$-optimal for all sufficiently
large~$c$ (Theorem~\ref{thm:final}).

\medskip

Our paper is organized as follows: In Section~\ref{sec:pps} we collect
some necessary preliminary remarks on periodic sets, in particular 
about their representations, their symmetries
and attached average theta series.
In Section~\ref{sec:two} we define the $f$-potential energy of a
periodic set and show how it can be expanded in the neighborhood of 
a given $m$-periodic representation. Section~\ref{sec:critical} gives necessary and
sufficient conditions for a periodic set to be an $f_c$-critical
configuration for all $c>0$. 
We provide a simplification for the expression of energy for
the special case of $2$-periodic sets in Section~\ref{sec:2periodic}. 
This can in particular be applied 
to the sets $\mathsf{D}^+_n$, which we describe in more
detail in Section~\ref{sec:four}.
In Section~\ref{sec:hessian} we obtain all necessary ingredients  
to show that $\mathsf{D}^+_n$ for odd $n\geq 9$ is locally
$f_c$-optimal for all sufficiently large $c$.
In our concluding Section~\ref{sec:conclusion} we also explain how this
result could possibly be extended, to prove at least locally a kind of
universal optimality of the set~$\mathsf{D}^+_9$.

\section{Preliminaries on periodic sets}\label{sec:pps}

We record in this section some preliminary remarks about periodic
sets. These may be of interest in their own, but will in particular 
be useful in subsequent computations. 
The first of these remarks is about minimal representations of periodic sets.

\begin{defn}\label{period} 
A \emph{periodic set} in $\R^n$ is a closed discrete subset $\Lambda$ of $\R^n$ which is invariant under translations by all the vectors of a full dimensional lattice $L$ in $\R^n$, that is
\begin{equation}\label{eqn:translat} 
\Lambda + L =\Lambda.
\end{equation}

A lattice for which \eqref{eqn:translat} holds is called a \emph{period lattice for $\Lambda$}.
\end{defn}

If \eqref{eqn:translat} holds, then the quotient $\Lambda/L$ is discrete and compact, hence finite. From this we can derive an alternative definition of a periodic set in $\R^n$, as a set of points which can be written as a union of finitely many cosets of a full-rank lattice $L$, \ie
\begin{equation} \label{eqn:periodic} 
\Lambda = \bigcup_{i=1}^m\left(  t_i + L\right)
\end{equation}
for some vectors $t_1, \dots, t_m$ in $\R^n$, which we assume to be pairwise incongruent $\mod L$. In that case we say that $\Lambda$ is $m$-periodic. 

Note that closedness is necessary in Definition~\ref{period}, as shown
by the counterexample 
$\Lambda=\bigcup_{n \in
  \mathbb{N}^*}\left(\frac{1}{n}+\mathbb{Z}\right)$ 
which is invariant under translations by $\Z$ but not of the
form~\eqref{eqn:periodic} for any $m$.

\subsection*{Representations}

We call the set of data, \ie a lattice $L$ together with a collection
$\mathbf{t}=\left(t_1,\ldots,t_m\right)$ of translational vectors, a
\textit{representation} of $\Lambda$, which we write
$\left(L,\mathbf{t}\right)$ for short. 
A given periodic set $\Lambda$ admits  infinitely many period lattices and representations, in which the number $m=\left\vert\Lambda\slash L\right \vert$ varies. For instance one can replace $L$ by any of its sublattice $L'$ and obtain a representation  as a union of  $m\left[L :L' \right]$ translates of $L'$, as in the example in Figure~\ref{fig:fig1}, where the same set is represented as a $4$ and $8$-periodic set.

\begin{figure}[h]\label{fig:fig1} 
        \centering
        \includegraphics[scale=0.1]{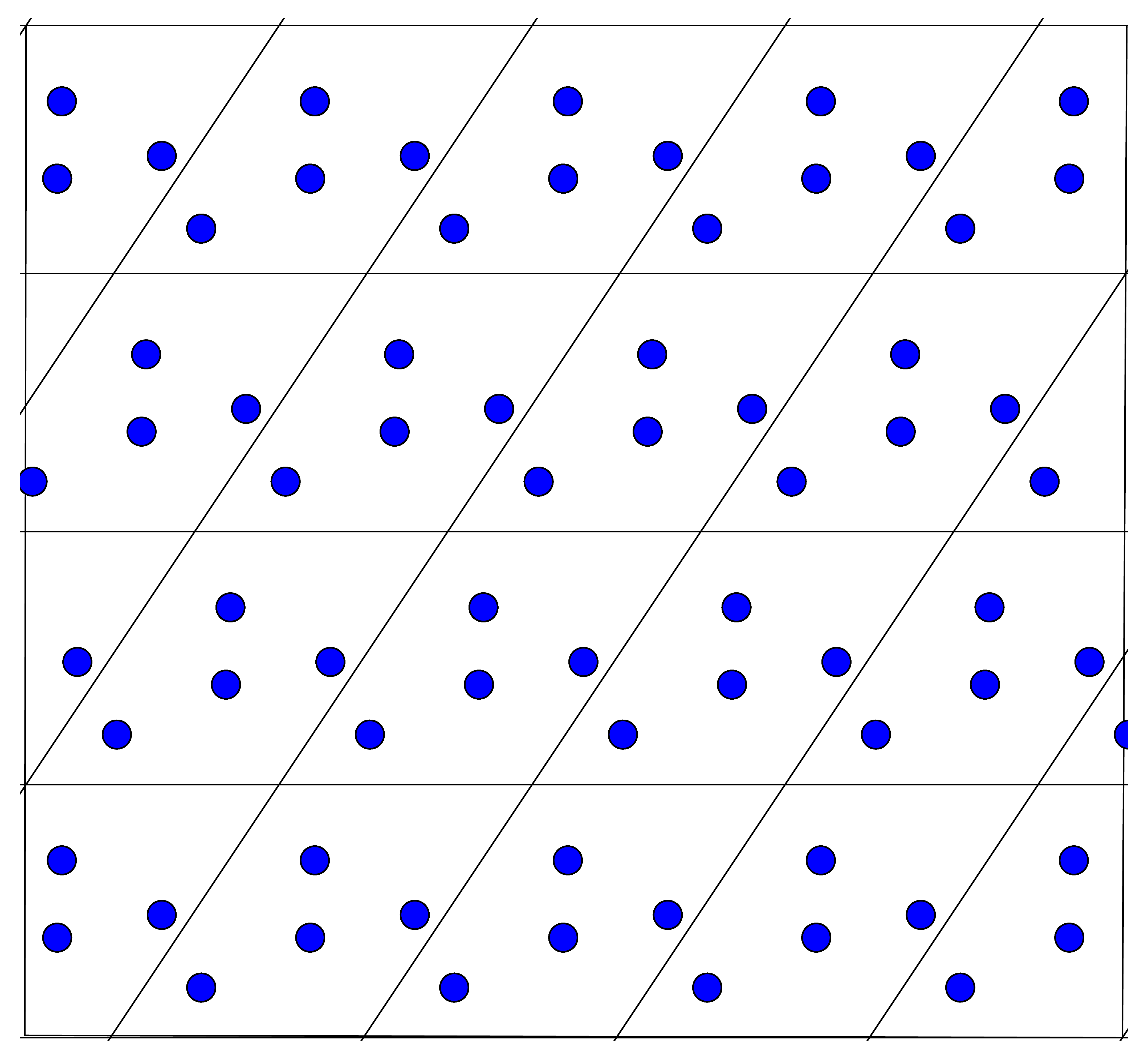} \hspace{1cm} \includegraphics[scale=0.1]{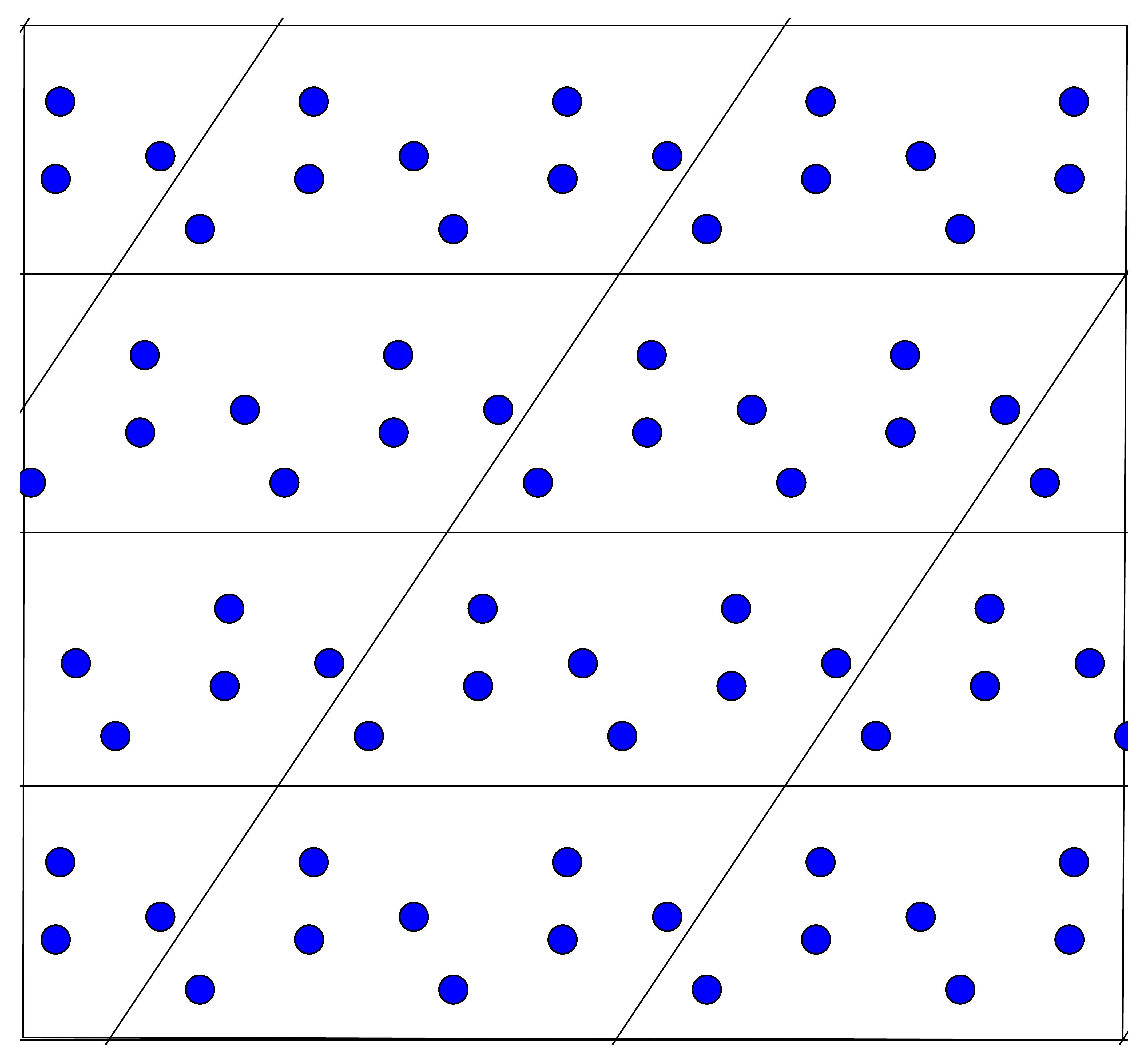}   
        \caption{}
\end{figure}

However, the set of period lattices, which is partially ordered by inclusion, admits a maximum $L_{\max}$, which we call the \emph{maximal period lattice of $\Lambda$} (see Proposition \ref{lmax} below), corresponding to an essentially unique \emph{minimal} representation of $\Lambda$ (\ie with a minimal number of cosets).

Note also that the \emph{point density} $p\delta(\Lambda)=\dfrac{m}{\sqrt{\det L}}$ of a periodic set $\Lambda$, which counts the "number of points per unit volume of space", does not depend on the choice of a representation. When studying properties which are invariant by scaling, we restrict to periodic sets with fixed point density.

We will be interested in quantities, such as energy, which depend only on the pairwise differences of elements of $\Lambda$ (see Definition \ref{pnrj} below). For any $x$ in $\Lambda$, we define the \textit{difference set $\Lambda_x$ of $x$} as the translate of $\Lambda$ by the vector $-x$: 
\begin{equation}
\Lambda_x \coloneqq\Lambda-x=\set{y-x \mid y \in \Lambda}.
\end{equation}
Two points $x$ and $y$ in $\Lambda$ have the same difference set if and only if $\Lambda$ is invariant under the translation by $x-y$. This is the case in particular if $x$ and $y$ are congruent modulo a period lattice of $\Lambda$.  
The following proposition shows that the number $m(\Lambda)$ of distinct difference sets $\Lambda_x$ as $x$ runs through $\Lambda$  is equal to the minimal number of cosets needed to represent $\Lambda$ as a periodic set, \ie the cardinality of the quotient of $\Lambda$ by its maximal period lattice:

\begin{proposition}\label{lmax}  Let $\Lambda$ be a periodic set in $\R^n$, and let $m(\Lambda)$ be the number of distinct difference sets $\Lambda_x$ as $x$ runs through $\Lambda$. Then the following holds:
\begin{enumerate}
\item\label{uno}  For every period lattice $L$ of $\Lambda$ one has 
\begin{equation*}
\left\vert\Lambda\slash L\right \vert \geq m(\Lambda)
\end{equation*}
with equality if and only if $L$ is maximal with respect to inclusion among period lattices of~$\Lambda$. 
\item\label{dos}  There exists a unique period lattice $L_{\max}$ containing all period lattices of $\Lambda$, defined as
$$ L_{\max}=\left\lbrace  v \in  \R^n \mid v + \Lambda = \Lambda \right\rbrace .$$
We call it the \emph{maximal period lattice} of $\Lambda$.
It corresponds to an essentially unique minimal representation of
$\Lambda$ as a union of 
$m(\Lambda) = \left \vert\Lambda\slash L_{\max}\right \vert $ 
translates of $L_{\max}$ (up to the choice of representatives modulo $L_{\max}$ and reordering). 
\item \label{tres} For $x$ and $y$ in $\Lambda$ one has
$$\Lambda_x = \Lambda_y \Leftrightarrow x \equiv y \!\!\mod L_{\max}.$$
\end{enumerate}
\end{proposition}
\begin{proof} 
\ref{uno}. As already noticed, two elements of $\Lambda$ which are congruent modulo a period lattice $L$ have the same difference set, so that $m(\Lambda)$ is at most $\left\vert\Lambda\slash L\right \vert$. If $L$ is not maximal, then there exists a period lattice $L'$ containing $L$ with finite index and we have 
$$ \left\vert\Lambda\slash L\right \vert = \left[L':L\right]\left\vert\Lambda\slash L'\right \vert \geq  \left[L':L\right] m(\Lambda) > m(\Lambda).$$
Conversely, if $\left\vert\Lambda\slash L\right \vert > m(\Lambda)$, then there are at least two elements $x$ and $y$ in $\Lambda$ which are not congruent modulo $L$ and have the same difference sets. Then  $\Lambda-x=\Lambda-y$, so that $\Lambda +(x-y)=\Lambda=\Lambda+(y-x)$ and more generally, $\Lambda$ is stable under translation by any vector in $\Z(x-y)$. The group $L'\coloneqq L +\Z(x-y)$ is discrete (it is contained in a translate of the discrete set $\Lambda$) hence a full dimensional lattice in $\R^n$  strictly containing $L$,  and since $\Lambda + L' =\Lambda$, it is indeed a period lattice of $\Lambda$.

\ref{dos}. Starting from any period lattice $L$, we can enlarge it using the construction described above as long as $\left\vert\Lambda\slash L\right \vert > m(\Lambda)$. The process ends up with a maximal period lattice. Since the sum $L+L'$ of two period lattices $L$ and $L'$ for $\Lambda$ is again a period lattice containing $L$ and $L'$, we see that such a maximal period lattice is unique, and contains all period lattices. It is also clear from its construction that it consists precisely of the vectors $v$ in the ambient space such that $v+\Lambda=\Lambda$.

\ref{tres}. This follows since $\Lambda_x = \Lambda_y \Leftrightarrow \Lambda-x = \Lambda-y \Leftrightarrow (x-y)+\Lambda=\Lambda$.

\end{proof}

For a given representation $\Lambda = \bigcup_{i=1}^m\left(t_i + L\right)$ of a periodic set $\Lambda$, the set "$\Lambda - \Lambda$" of pairwise differences of elements of $\Lambda$ can be described as 
$$\Lambda - \Lambda= \bigcup_{1 \leq i \leq m}\Lambda_{t_i}.$$
As an ordinary set, it does not depend on the choice of a representation $\left(L,\mathbf{t}\right)$, but it does as a "multiset", since the difference of two elements of $\Lambda$ may occur in several difference sets $\Lambda_{t_i}$. Moreover, the number of difference sets to which a given element of $\Lambda - \Lambda$ belongs depends on the representation chosen.
To eliminate this dependency, we define a weight function $\nu$ on $\Lambda - \Lambda$, setting
\begin{equation}
\nu(w)=\frac{1}{m}\left| \left\lbrace i \mid w \in \Lambda_{t_i}\right\rbrace\right|.
\end{equation}
This definition is independent of the choice of a representation of $\Lambda$, namely one has
\begin{equation}\label{poids} 
\nu(w)=\frac{1}{m(\Lambda)}\left| \left\lbrace x \in \mathcal{R}\mid w
    \in \Lambda_x\right\rbrace\right|
,
\end{equation}
where $\mathcal{R}$ is a set of representatives of $\Lambda \!\!\mod L_{\max}$.

Note also that $\nu(w)=1$ if and only if  $w \in L_{\max}$. Indeed,
$w$ has weight $1$ if and only if it belongs to all difference sets
$\Lambda_x$: It is clearly the case if $w \in L_{\max}$, and
conversely, if $w$ belongs to 
$\displaystyle \bigcap_{x \in \mathcal{R}}\Lambda_x$, 
then there exists a permutation $\sigma$ of $\mathcal{R}$ such that 
$$\forall x \in \mathcal{R}, \  w +x \equiv \sigma(x) \!\!\mod L_{\max}$$
which implies that $w+\Lambda=\Lambda$, so that $w \in L_{\max}$.

Note also, in the same spirit, the following two observations: 
\begin{itemize}
\item if $m(\Lambda)=1$, \ie if $\Lambda$ is a translate of a lattice, then one has $\nu(w)= 1$ for all $w \in \Lambda - \Lambda=\Lambda$.
\item if $m(\Lambda)=2$, then one has $\nu(w)= 1$ or $\frac{1}{2}$
  according to $w$ belonging to the maximal period lattice of $\Lambda$ or not.
\end{itemize}

\subsection*{Symmetries}

We continue this preliminary section with some considerations on
automorphisms. 
To a lattice $L$ in $\R^n$ one associates the group $\Aut L$ of its 
\emph{linear} automorphisms defined as
\begin{equation}
\Aut L=\left\lbrace f \in O(\R^n) \mid f(L) =L\right\rbrace.
\end{equation}
 For a more general periodic set $\Lambda$, the natural group of
 transformations to consider is the group 
$\Isom\Lambda$ of \emph{affine isometries} preserving it. 
If $f$ is such an affine isometry, then its associated orthogonal 
 automorphism $\bar{f}$, defined by the property that 
 $\bar{f}(x-y)=f(x)-f(y)$ for all $x$ and $y$ in $\R^n$, stabilizes the 
 maximal period lattice $L_{\max}$. Indeed, for every $\ell \in L_{\max}$, one has 
 $$\bar{f}(\ell)+\Lambda=\bar{f}(\ell)+f(\Lambda)=f(\ell+\Lambda)=f(\Lambda)=\Lambda,$$
  whence $\bar{f}(\ell) \in L_{\max}$, by the very definition of $L_{\max}$. 
  
 We denote by $\Aut \Lambda$ the image of $\Isom \Lambda$ in $\Aut L_{\max}$, \ie  the subgroup of $\Aut L_{\max}$ consisting of all maps 
 $\bar{f}$ as $f$ runs through $\Isom \Lambda$, and call it the
\emph{group of orthogonal automorphisms of $\Lambda$}.
%

Two affine isometries of $\Lambda$ with the same associated 
orthogonal automorphism  $\bar{f}$ differ by a translation by a vector in 
$L_{\max}$. Therefore, we get the following short exact sequence
\begin{equation}
\begin{array}{rcrlcl}
1 \longrightarrow & L_{\max} \longrightarrow & \Isom \Lambda \longrightarrow &\Aut \Lambda \longrightarrow &1 \\
& & f \longmapsto & \bar{f} & 
\end{array}
\end{equation}
which is no split in general (it is split for instance when $\Lambda$ is a lattice). Disregarding translations by $L_{\max}$, the main object of interest is thus the group $\Aut \Lambda$ of orthogonal automorphisms which we now characterize:

\begin{lemma}\label{auto} 
Let $\Lambda= \bigcup_{i=1}^{m} t_i + L_{\max}$ be an $m$-periodic set in $\R^n$ given by 
a minimal representation. Let $\Isom \Lambda$ be the group of its
affine isometries and 
$\Aut \Lambda=\left\lbrace \bar{f} \mid f \in \Isom \Lambda
\right\rbrace \subseteq \Aut L_{\max}$ 
be the group of its orthogonal automorphisms. Then:

\begin{enumerate}
	\item\label{perm} For every $f \in \Isom \Lambda$ there exists a unique
	permutation 
	$\sigma \in \mathfrak{S}_m$ such that 
	\begin{equation*}
	f(t_i) \equiv t_{\sigma(i)} \!\!\mod L_{\max} \text{ for all } i \in \left\lbrace 1, \dots, m 
	\right\rbrace.
	\end{equation*}
	\item \label{aut} An element $\varphi \in \Aut L_{\max}$ belongs to $\Aut \Lambda$ if and only if
	\begin{equation}\label{orthaut}
\exists \sigma \in \mathfrak{S}_m \text{ s.t. } \varphi(t_i - t_1)  \equiv t_{\sigma(i)}-t_{\sigma(1)} \!\!\mod L_{\max} \text{ for all } i \in \left\lbrace 1, \dots, m 
\right\rbrace
	\end{equation}
	in which case it is associated to the affine isometry $x \mapsto \varphi(x-t_1)+t_{\sigma(1)}$.
 
\end{enumerate}
 
\end{lemma}
\begin{proof}
If $f \in \Isom \Lambda$, then for each $i \in \left\lbrace 1, \dots, m 
\right\rbrace$ there exists an index 
$\sigma(i)$  such that $f(t_i) \in t_{\sigma(i)} + L_{\max}$, and $\sigma$ is a 
permutation since
$$f(t_i) \equiv f(t_j) \!\!\mod L_{\max}\Leftrightarrow \bar{f}(t_i-t_j) \in 
L_{\max} \Leftrightarrow t_i-t_j \in L_{\max} \Leftrightarrow i=j.$$ 
This proves~\ref{perm} as well as the congruence $\bar{f}(t_i - t_1)  \equiv t_{\sigma(i)}-t_{\sigma(1)} 
\!\!\mod L_{\max}$ for all $i\in \left\lbrace 1, \dots, m 
\right\rbrace$. Conversely, if $\varphi \in \Aut 
L_{\max}$ satisfies \eqref{orthaut} for some permutation $\sigma$, then the map $f_{\varphi}(x):= \varphi 
(x-t_1)+t_{\sigma(1)}$ is in $\Isom \Lambda$ and 
$\bar{f_{\varphi}}=\varphi$, which establishes \ref{aut}.
\end{proof}

Note that for each $\varphi$ in $\Aut \Lambda$, the associated permutation 
$\sigma$ is unique, as a consequence of the maximality of $L_{\max}$: 
If $\sigma$ and $\gamma$ are two permutations of $\Lambda \slash L_{\max}$ 
such that $\varphi(t_i - t_1)  \equiv  t_{\sigma(i)}-t_{\sigma(1)} \equiv 
t_{\gamma(i)}-t_{\gamma(1)}\mod L_{\max}$ for all $i$, then 
$t_{\sigma(i)}-t_{\gamma(i)} \equiv  t_{\sigma(1)}-t_{\gamma(1)}=: u\mod L_{\max}$ 
for all  $i$, whence $t_{\sigma(i)}\equiv u+t_{\gamma(i)}$, which implies that $u \in 
L_{\max}$, so that $\sigma=\gamma$. 

Also, the elements of $\Aut \Lambda$ stabilize the set $\Lambda - \Lambda= \bigcup_{1 \leq i \leq m}\Lambda_{t_i}$. More precisely, for each $\varphi \in \Aut \Lambda$ one has $\varphi (\Lambda_{t_i})=\Lambda_{t_{\sigma(i)}}$ where $\sigma$ is the permutation of $\Lambda \slash L_{\max}$ canonically associated to $\varphi$. This last property makes this group the right object to consider in the sequel.

\begin{rem} \label{rem:auto-discuss} 
For a given periodic set 
$\Lambda = \bigcup_{i=1}^m\left(t_i + L_{\max}\right)$, 
we can often assume 
without loss of generality that $t_1=0$ (it amounts to translate
$\Lambda$ by a fixed vector). In such a situation, $\Aut \Lambda$
contains, with index at most $m$,  the subgroup $$\Aut_{0} \Lambda =
\left\lbrace \varphi  \in \Aut L_{\max} \mid \varphi(\Lambda)=\Lambda
\right\rbrace.$$ This corresponds to permutations $\sigma$  fixing $1$
in \eqref{orthaut} and could be a natural choice for an alternative
definition of the group of automorphisms of $\Lambda$. Nevertheless,
it would introduce a somewhat unnecessary dissymmetry between the
$t_i$'s, and would lead to disregard some automorphisms which are
natural to consider.
\end{rem}

For example, for a $2$-periodic set
 $$\Lambda=L_{\max} \cup \left(  v+L_{\max}\right) \ , \ 2v\notin L_{\max},$$ 
we have $-\id \in \Aut \Lambda \setminus \Aut_{0} \Lambda$ and $\left[ \Aut \Lambda : \Aut_0\Lambda\right]=2$. 
 
At the other end, if $\Lambda$ is a $3$-periodic set of the form 
$$\Lambda=L_{\max} \cup \left(v+L_{\max}\right)\cup\left( -v+L_{\max}\right) \ , \ 2v\notin L_{\max},$$ 
 then one checks that $\Aut \Lambda = \Aut_0\Lambda$.

\subsection*{Review on theta series and modular forms}

For some estimates needed in Section~\ref{sec:plc} we use
certain theta series and their properties, which we review here. 
To start with, we state a rather general result about the modularity of theta series with spherical coefficients attached to a rational periodic set.

If $L$ is a lattice in $\R^n$ and $\rho$ is any vector in 
$\R^n$, one defines, for $z$ in the upper half-plane $\mathbb H =\left\lbrace z \in 
\C \mid \Im z >0 \right\rbrace$
\begin{equation}\label{theta}
\theta_{\rho+L}(z) =\sum\limits_{x \in \rho + L} e\left(\dfrac{\Vert 
	x \Vert^2 z}{2}\right)
\end{equation}
where $e(z)=e^{2\pi i z}$. When $\rho=0$, this reduces to the standard theta series 
of the lattice $L$. 

As in the lattice case, one can introduce \emph{spherical 
	coefficients} in the previous definition, namely,  if $P$ is a harmonic polynomial, one 
defines 
\begin{equation}\label{thetaspher}
\theta_{\rho+L,P}(z) =\sum\limits_{x \in \rho + L} P(x)e\left(\dfrac{\Vert 
	x \Vert^2 z}{2}\right).
\end{equation}
From this, and following \cite{MR688626}, we define the \emph{average theta series} with spherical coefficients $P$ of a periodic set $\Lambda= \bigcup_{i=1}^m\left(  t_i + L\right)$ as

\begin{align*}
\theta_{\Lambda,P}(\tau)&=\dfrac{1}{m}\sum\limits_{1\leq i,j\leq m} \theta_{t_i-t_j+L,P}(\tau)\\
&=\theta_{L,P}(\tau) + \dfrac{2}{m}\sum\limits_{1\leq i < j\leq m} \theta_{t_i-t_j+L,P}(\tau).
\end{align*}

Both, \eqref{theta} and \eqref{thetaspher}, satisfy transformation formulas under 
$\SL(2,\Z)$, from which one deduces,  
under suitable assumptions on $L$ and $\rho$, that $\theta_{\rho+L,P}(z)$ (resp. $\theta_{\Lambda,P}$) is a \emph{modular form} for some modular group and character (see Proposition \ref{mod} below).
 Let $L$ be an even integral lattice, \ie $x \cdot x$ is even for all $x \in L$. 
The level of $L$ is the smallest integer $N$ such that $\sqrt{N}L^{*}$ is even integral (this implies in particular that $NL^{*} \subseteq L$).
\begin{proposition}\label{mod}
	Let $L$ be an even integral lattice of dimension $n$ and level
        $N$. Then, for any $\rho \in L^{*}$, and any spherical
        harmonic polynomial $P$ of degree $k$, 
        the theta series $\theta_{\rho+L,P}(z)$ is a modular form of weight $k+\dfrac{n}{2}$ for the principal congruence group $$\Gamma (4N)=\left\lbrace \tau =\begin{pmatrix}
	a & b \\ 
	c & d
	\end{pmatrix} \in \SL(2,\Z) \mid \tau \equiv \begin{pmatrix}
	1 & 0\\ 
	0 & 1
	\end{pmatrix}\mod 4N\right\rbrace $$ and the character $$\vartheta\left(\tau \right) = \left(\dfrac{2c}{d}\right)^n
	.$$
	Moreover, if $k>0$, then $\theta_{\rho+L,P}(z)$ is a cusp form.
\end{proposition}
\begin{proof}
	This is essentially \cite[Corollary 10.7]{MR1474964}, up to
        reformulation: setting $L=g \Z^n$ for some $g \in \GL(n,\R)$,
        $A=g^tg$, 
and $h=Ng^{-1}(\rho)$, the condition $\rho \in L^{*}$ is equivalent to
$Ah \equiv 0 \mod N$ (the condition defining the set $\mathcal H$ in
\cite[Corollary 10.7]{MR1474964}) and $\theta_{\rho+L,P}(z)$ coincides
with the congruence theta series $\Theta(z;h)$ in the above reference,
whence the conclusion follows.
\end{proof}

\section{Energy of periodic sets}\label{sec:two} 

We recall in this section some basic facts about the energy of a periodic set and its local study, which were established in \cite{MR2889159}.

Following Cohn and Kumar \cite{MR2257398} we define the energy of a periodic set with respect to a non negative potential function as follows:
\begin{defn}\label{pnrj} 
Let $\Lambda$ be a periodic set with maximal period lattice $L_{\max}$, and $f$ a non-negative potential function. We set 
\begin{equation}\label{eqn:nrj} 
E(f,\Lambda)=\frac{1}{m(\Lambda)}\sum_{x \in \mathcal{R}}\sum_{\substack{u\in \Lambda_x\\ u \neq 0}}f(\norm{u}^2)
\end{equation}
where $\mathcal R$ is a set of representatives of $\Lambda$ modulo $L_{\max}$.
\end{defn}

This sum may diverge, in which case  the energy is \emph{infinite}. 
Note that if $\Lambda$ is given by an $m$-periodic representation $\Lambda = \bigcup_{i=1}^m x_i + L$, non necessarily minimal, then one has 
 $$E(f,\Lambda)=\frac{1}{m}\sum_{i=1}^m\sum_{\substack{u \in \Lambda _{x_i}\\ u \neq 0}} f(\norm{u}^2)=\frac{1}{m}\sum_{1 \leq i, j \leq m}\sum_{\substack{w \in L \\ w+x_j-x_i \neq 0}} f(\norm{w+x_j-x_i}^2)$$
in accordance with the definition used in \cite{MR2257398}.
 
This "non intrinsic" formulation
$\frac{1}{m}\sum_{i=1}^m\sum_{\substack{u \in \Lambda _{t_i}\\ u \neq
    0}} f(\norm{u}^2)$ is often better suited for explicit computations because it allows to use representations of periodic sets that are not assumed to be minimal.

We want to expand the $f$-energy in a neighbourhood of a given $m$-periodic set $$\Lambda_0=\bigcup_{i=1}^m t^0_i+L_0.$$
Note that the question of periodic sets with minimal $f$-energy
(with $f$ being monotone decreasing) only makes sense if we restrict to 
periodic sets with a fixed point density. 
Otherwise, the energy can be
made arbitrary small by scaling. 
So we restrict to $m$-periodic sets $\Lambda$ with fixed point density, \ie of the form
$$\Lambda =\displaystyle\bigcup_{i=1}^m (t^0_i+t_i)+gL_0$$
with $g \in \SL(n,\R)$.

As in \cite[\S 3]{MR2889159}, we set
$g^{t}g=A_0^{t}A_0\exp\left(A^{-1}_0HA_0\right)$ 
where $L_0=A_0\Z^n$ and $H$ is a trace zero symmetric matrix. Then the evaluation of the energy $E(f,\Lambda)$ as $\Lambda$ varies in a neighbourhood of the initial periodic set $\Lambda_0$ reduces to the study of the quantity

\begin{equation}\label{nrg3} 
E_f(H,\mathbf{t}) \coloneqq\frac{1}{m}  \; \sum_{1\leq i,j\leq m} \; \sum_{0\neq w \in t^0_i-t^0_j+L_0}  f(\exp(H)\br{w+t_i-t_j})
\end{equation}
for small enough $\mathbf{t}\in \R^{mn}$ and $H \in\mathcal T:=\set{Q \in \Sn \mid \Tr(Q)=0}$, where $\Sn$ stands for the space of $n\times n$ real symmetric matrices (see \cite[\S 3]{MR2889159} for details).

Using the Taylor expansion of the matrix exponential we write
\begin{equation*}
\exp(H)[w+t_i-t_{j}]={\|w\|^2}+\mathcal{L}\pr{H,\mathbf{t}}+\mathcal{S}\pr{H,\mathbf{t}}+ o (\|\pr{H,\mathbf{t}}\|^2)
\end{equation*}
where
\begin{equation*}
\mathcal{L}\pr{H,\mathbf{t}}=H[w]+ 2w^{t}(t_i-t_{j})
\end{equation*}
and
\begin{equation*}
\mathcal{S}\pr{H,\mathbf{t}}=\|t_i-t_{j}\|^2+2w^{t}H(t_i-t_{j})+\frac{1}{2}H^2\br{w}.
\end{equation*}

In particular, if $f(r)=e^{-cr}$ we get 
\begin{equation*}
e^{-c\exp(H)[w+t_i-t_{j}]}=e^{-c\|w\|^2}\left(1-c\left(\mathcal{L} +\mathcal{S} \right)+\frac{c^2}{2}\mathcal{L}^2\right)+ o (\|\pr{H,\mathbf{t}}\|^2) 
\end{equation*}
and hence the following expressions for the gradient 
\begin{equation*}
 \grad = -\frac{c}{m} 
\sum_{1\leq i,j\leq m}\sum_{0\neq w \in t^0_i-t^0_j+L_0}
\left(   
H[w]+ 2w^{t}(t_i-t_j) 
\right) e^{-c \|w\|^2} 
\end{equation*}
and the Hessian
\begin{eqnarray*}
\hess&=&\frac{c}{m} 
\sum_{1\leq i,j\leq m}\sum_{0\neq w \in t^0_i-t^0_j+L_0}
\left( \left\lbrace \frac{c}{2}H[w]^2-\frac{1}{2}H^2\br{w}\right\rbrace
\right.  
\\&+&
\left.  
\left\lbrace 2c \left( w^{t}(t_i-t_j)\right)^2  - \|t_i-t_j\|^2-2w^{t}H(t_i-t_j)+2cw^{t}(t_i-t_j)H[w]\right\rbrace 
\right) e^{-c \|w\|^2}  
.
\end{eqnarray*}

\section{Critical Points}   \label{sec:critical}
A periodic set is said to be $f$-critical if it is a critical point for the energy $E_f$. We will be especially interested in $f_c$-critical periodic sets, where~$f_c(x)=e^{-cx}$ with~$c>0$, since these functions generate the space of completely monotonic functions (see \cite[Theorem 12b, p. 161]{MR0005923}).

We want to give a necessary and sufficient criterion for 
a periodic set $\Lambda$ in $\R^n$ to be $f_c$-critical for all $c>0$. Using the formulas of the previous section this amounts to show that the gradient vanishes for all choices of~$c>0$. 

Collecting the terms in the sum above with the same value $e^{-c\|w\|^2}$, we obtain the following:

\begin{lemma}\label{lem:sums}
A periodic set~$\Lambda$ in $\R^n$ is $f_c$-critical for all~$c>0$ if and only if the terms
$$
\sum_{1\leq i,j\leq m} 
\; \sum_{ w \in t^0_i-t^0_j+L_0 , \; \|w\| = r}  \;
H[w]+ 2w^{t}(t_i-t_j) 
$$
vanish for any representation $\Lambda=\bigcup_{i=1}^m t^0_i+L_0$
and any choice of $r>0$ and $(H,\mathbf{t})$. 
\end{lemma}

\begin{proof} According to the previous section, the gradient of $E_{f_c}$can be written as
\begin{equation}
\grad = -\frac{c}{m} 
\sum_{r>0} \left[\sum_{1\leq i,j\leq m}\;\sum_{ w \in t^0_i-t^0_j+L_0, \|w\|=r}
\left(  
H[w]+ 2w^{t}(t_i-t_j) 
\right)\right] e^{-c r^2}
\end{equation}
Suppose there is a representation of~$\Lambda$ 
and a minimal $r>0$ for which the sum between brackets does 
not vanish for some choice of $(H,\mathbf{t})$.
Then for sufficiently large~$c$ the gradient is 
essentially given by the corresponding term
(in front of $e^{-c r^2}$). So the gradient does not vanish as well.

If on the other hand the gradient vanishes for all $c$,
we find that the corresponding sums of the proposition 
all have to vanish.
\end{proof}

We want to state a necessary and sufficient condition for the vanishing
of all the sums of the previous propositions in terms of \emph{weighted spherical designs}.
For a periodic set $\Lambda$, $x\in \Lambda$ and $r >0$ we define
$$
\Lambda_x(r) \; = \;  \{ y-x  \mid \| y-x \| = r, \; y \in \Lambda \}
$$
and we set $\Lambda(r) = \bigcup_{x\in \Lambda} \Lambda_x(r)$.

A \emph{weighted spherical $t$-design} is a pair $(X,\nu)$ of a
finite set $X$ contained in a sphere of radius~$r$
and a weight function $\nu$ on $X$ such that 
\begin{equation}
\dfrac{1}{\left| S \right|}\int_S f(x) dx = \dfrac{1}{\left| X \right|}\sum_{x \in X} \nu(x) f(x)
\end{equation}
for all polynomials $f(x)=f(x_1, \dots,x_n)$ of degree at most $t$. This is a special case of a \emph{cubature formula} on the sphere, studied \eg by Goethals and Seidel in \cite{MR661779}, and reduces to the classical notion of spherical $t$-design when the weight function is equal to $1$. 

Note that, for $t=1$, this simply means that the weighted sum $\sum_{x \in X} \nu(x) x$ is $0$. When all weights are $1$, this reduces to the condition 
$$\sum_{x \in X} x = 0$$ 
which we refer to in the sequel as $X$ being a \emph{balanced} set. One may think of forces acting on the origin that balance each other.

Finally, we mention the following useful characterization of the $2$-design property, which will be used throughout the rest of the paper:

\begin{lemma}[\cite{MR964837}, Theorem 4.3]\label{NS}\quad
A weighted set $(X,\nu)$ on a sphere of radius $r$ in $\R^n$ is a weighted spherical $2$-design if and only if 
$$\sum_{x \in X} \nu(x) x =0 \text{ and } \sum_{x \in X} \nu(x) x x^t=c\id_n$$ for some constant $c$.
\end{lemma}

\begin{theorem} \label{th:periodic-critical}
A periodic set~$\Lambda$ in $\R^n$ is $f_c$-critical for any~$c>0$ if and only if
\begin{enumerate}
\item
All non-empty shells $\Lambda_x(r)$ for $x\in \Lambda$ and $r>0$ are balanced.
\item
All non-empty shells $\Lambda(r)$ for $r>0$ are weighted spherical $2$-designs with respect to the weight $\nu(w)=\frac{1}{m(\Lambda)}\left| \left\lbrace x \in \Lambda \!\!\mod L_{\max} \mid w \in \Lambda_x\right\rbrace\right|$.
\end{enumerate}
\end{theorem}

Note that the statement of the theorem, in contrast to the one of Lemma~\ref{lem:sums},
is independent of the possible representations of~$\Lambda$.

\begin{proof}
First observe that the sums of Lemma~\ref{lem:sums}
split
for any representation $\Lambda=\bigcup_{i=1}^m t^0_i+L_0$
into two parts: one depending on $H$ only and one depending on $\mathbf{t}$ only.

The part depending on $\mathbf{t}$ is (up to a factor of $2$) equal to
$$
\sum_{1\leq i,j\leq m} 
\; \sum_{ w \in t^0_i-t^0_j+L_0, \; \|w\| = r}  \;
w^t (t_i - t_j)
.
$$
We can rearrange the sum, collecting terms that occur with a fixed $t_k$,
either for $k=i$ as $u=w \in (t^0_i +L_0) - (t^0_j+L_0)$
or for $k=j$ as $u=-w \in (t^0_j +L_0) - (t^0_i+L_0)$ and get

$$
-2\sum_{1\leq k\leq m}  
\left(
\; \sum_{ u \in \Lambda_{t^0_k}(r)}  \;
u^t
\right)t_k.
$$
So this sum vanishes for all choices of $\mathbf{t}$ if and only if 
the coefficients of each $t_k$ vanish.
This is precisely the case if and only if $\Lambda_x(r)$ is balanced
for every $x\in \Lambda$. This implies that $\Lambda(r)$ itself is a
weighted balanced set (weighted spherical $1$-design) since
\begin{equation*}
\sum_{x \in \mathcal R} \sum_{u\in \Lambda_x(r)}u=\sum_{u\in
  \Lambda(r)}\nu(u) u,
\end{equation*}
with $\mathcal{R}$ being a set of representatives of $\Lambda \!\!\mod L_0$.

The part depending on $H$ can be rewritten as 
\begin{align*}
\sum_{1\leq i,j\leq m} 
\; \sum_{ w \in t^0_i-t^0_j+L_0 , \; \|w\| = r}  \;
H[w]
&\;= \;
\left\langle
H , 
\sum_{1\leq i,j\leq m} 
\; \sum_{ w \in t^0_i-t^0_j+L_0 , \; \|w\| = r}  \;
w w^t
\right\rangle\\ 
&\;= \;\left\langle
H , 
\sum_{1\leq j \leq m}\;\sum_{w \in \Lambda_{t^0_j}(r)} 
\; 
w w^t
\right\rangle\\
&\;= \;m\left\langle
H , 
\sum_{w \in \Lambda(r)} 
  \; \nu(w)
w w^t
\right\rangle
.
\end{align*}
It vanishes for all choices of trace zero symmetric matrice $H$ if and only if the sum of rank-$1$ forms (matrices) $\nu(w)w w^t$ is a (positive) multiple of the identity, namely
\begin{equation}\label{2des} 
\sum_{w \in \Lambda(r)} \nu(w) w w^t = c_r \id_n
\end{equation}
 with 
$$c_r=\dfrac{r^2\sum_{w \in \Lambda(r)} \nu(w)}{n},$$ 
where the value of the constant $c_r$ is obtained by taking the trace of \eqref{2des}.
Combined with the first part of the theorem which insures that $\Lambda(r)$ is already a weighted spherical $1$-design, this last condition is equivalent to $\Lambda(r)$ being a weighted spherical $2$-design, due to Lemma \ref{NS}.

\end{proof}

\section{Expressing energy of $2$-periodic sets}  \label{sec:2periodic}

In order to deal with the energy of $\mathsf{D}_n^+$ and more general
for other $2$-periodic sets, a reordering of contributing terms 
will be very helpful.

Let $\Lambda$ be a periodic set. Without loss of generality, we 
can assume that $\Lambda$ contains $0$ (it amounts to translate $\Lambda$ by a 
well-chosen vector). Note that this is equivalent to the property that $\Lambda$ 
contains its maximal period lattice $L_{\max}$. If we assume moreover that  
$m(\Lambda)=2$, then we have
$\Lambda=  
L_{\max} 
\cup \left(v +  L_{\max}\right)$ for any $v \in \Lambda \setminus  L_{\max}$ and 
 
$$\Lambda_x=\begin{cases}
\Lambda \text{ if } x \in L_{\max}\\
-\Lambda \text{ otherwise.}
\end{cases}$$
In particular, $\Lambda -\Lambda =\Lambda \cup \left( -\Lambda\right)=L_{\max} 
\cup \left(v +  L_{\max}\right)\cup \left(-v +  L_{\max}\right)$. The next lemma clarifies the consequences of these properties on a non-minimal representation of $\Lambda$.
\begin{lemma}\label{lem:perm} 
	Let $\Lambda =\displaystyle\bigcup_{i=1}^m t_i+L \subset \R ^n$ be a periodic set 
	\textbf{containing} $\mathbf{0}$. Suppose $m(\Lambda)=2$. Then there is a partition of $I=\left\lbrace 1, \ldots , m \right\rbrace$ into two equipotent subsets $J$ and $J'$ and a map $\sigma : I \times I \rightarrow I$ such that 
			\[
			\displaystyle
			\forall i \in \set{1,\dots,m} \quad t_i-t_k \equiv
			\left\{
			\begin{array}{ll}
			\phantom{-}t_{\sigma(i,k)} \!\!\mod L \; \mbox{ if} \; k\in J\\
			-t_{\sigma(i,k)} \!\!\mod L \; \mbox{ if} \; k\in J'
			\end{array}
			\right.
			.
			\]
			Moreover, for any fixed $i$ or $k$ in $J$
                        (resp. in $J'$), the maps $\sigma(i,\cdot)$
                        and $\sigma(\cdot,k)$ bijectively map $J$ onto
                        $J$ and $J'$ onto $J'$ (resp $J$ onto $J'$ and $J'$ onto $J$).
 
\end{lemma}

\begin{proof}
	If $m(\Lambda)=2$ then the maximal period lattice $L_{\max}$ of $\Lambda$ contains $L$ with index $\frac{m}{2}$ and, as mentioned above, if $v$ is any element in $\Lambda \setminus  L_{\max}$, one has
	$$\Lambda =\bigcup_{i=1}^m t_i+L =  L_{\max} \cup \left(v +  L_{\max}\right).$$ Consequently,  $t_i \in L_{\max}$ for exactly one half of the indices $ i\in I$ and $\Lambda-t_i = \Lambda$ or $-\Lambda$ according as $t_i $ belongs to $L_{\max}$ or $v +  L_{\max}$. Setting $J=\left\lbrace i \in I \mid t_i \in L_{\max}\right\rbrace$ and $J'=\left\lbrace i \in I \mid t_i \in v +  L_{\max}\right\rbrace$ one can construct the map $\sigma$ as follows:
	\begin{itemize}
		\item If $k \in J$, that is to say if  $t_k \in   L_{\max}$, then $\Lambda-t_k = \Lambda=\bigcup_{i=1}^m t_i+L$ so that  for all $i \in I$ there is a well-defined index $\sigma(i,k) \in I$ such that $t_i-t_k \equiv t_{\sigma(i,k)} \!\!\mod L$. On the other hand, since $\Lambda-t_k = \Lambda= L_{\max} \cup \left(v +  L_{\max}\right)$, we infer that $t_i-t_k$ belongs to $ L_{\max}$ or  $v +  L_{\max}$, depending on whether $t_i$ is  in $ L_{\max}$ ($\Leftrightarrow i\in J$) or $t_i$ is in $v +  L_{\max}$ ($\Leftrightarrow i\in J'$), which means that $\sigma(\cdot,k)$ maps $J$ to $J$ and $J'$ to $J'$. The injectivity of $\sigma(\cdot,k)$ is straightforward, as the $t_i$'s are noncongruent $\mod L$.
		\item If $k \in J'$, then $t_k \in v +   L_{\max}$ and $\Lambda-t_k = -\Lambda$ so that  for all $i \in I$ there is a well-defined index $\sigma(i,k) \in I$ such that $t_i-t_k \equiv -t_{\sigma(i,k)} \!\!\mod L$. Now, since $\Lambda-t_k = -\Lambda=- L_{\max} \cup -\left(v +  L_{\max}\right)$ we have this time that $t_i-t_k$ belongs to $- L_{\max}$ or $-\left(v +  L_{\max}\right)$ according to $i$ being in $J'$ or $J$, which means that $\sigma(i,k)$ belongs to $J$ if $i\in J'$ and to $J'$ if $i\in J$. Again, the injectivity of $\sigma(\cdot,k)$ is clear. 
	\end{itemize}
	It remains to prove that, for fixed $i$, the map $\sigma(i,\cdot)$ also satisfies the required properties, which proceeds by an easy case by case verification, as above. 
\end{proof}

Using the results of Section \ref{sec:two}, we know that in a suitable
neighborhood of our given set $\Lambda_0=\bigcup_{i=1}^m (t^0_i +
L_0)$, the $f$-energy varies according to 
\begin{equation}\label{eqn:local_energy} 
E_f(H,t) \coloneqq\frac{1}{m}  
\; \sum_{1\leq i,j\leq m} \; \sum_{0\neq w \in t^0_i-t^0_j+L_0}  f(\exp(H)[w+t_i-t_j])
\end{equation}
In what follows we will extensively use the following reordering of contributions:

\begin{lemma} \label{lem:ReorderingLemma}
Suppose $\Lambda_0=\bigcup_{i=1}^m (t^0_i + L_0) =  L_{\max} \cup (v+ L_{\max})$ 
with $v \in\R^n$ and lattice $L_0 \subseteq L_{\max}\subset \R^n$ is a $2$-periodic set, and that
$t^0_i\in L_{\max}$ for $i\in J=\{1,\ldots,\frac{m}{2}\}$, $t^0_i \in v + L_{\max}$ for $i\in J'=\{\frac{m}{2}+1,\ldots, m\}$.
Then
\begin{eqnarray*}
E_f(H,t)
& = &
\frac{2}{m^2}\left[
\left(
 \sum_{0\not= w \in  L_{\max}} \; \sum_{i=1}^m \; \sum_{k\in J} \;  
     f\left( \exp(H) [ w+t_i-t_{\sigma(i,k)}]\right)
\right)
\right. \\
& & \qquad +
\left(
\sum_{0\not= w \in -(v+ L_{\max})} \; \sum_{i\in J} \; \sum_{k\in J'} \;   
     f\left( \exp(H) [ w+t_i-t_{\sigma(i,k)}]\right)
\right) \\
& & \qquad +
\left.
\left(
\sum_{0\not= w \in (v+ L_{\max})} \; \sum_{i\in J'} \; \sum_{k\in J'} \;   
     f\left( \exp(H) [ w+t_i-t_{\sigma(i,k)}]\right)
\right) 
\right]
,
\end{eqnarray*}
where $\sigma(i,k)$ is defined as in Lemma~\ref{lem:perm}, that is
\[
\displaystyle
t^0_i-t^0_k \equiv_{\!\!\mod L_0} 
\left\{
\begin{array}{ll}
\phantom{-}t^0_{\sigma(i,k)} \; \mbox{if} \; k\in J\\
-t^0_{\sigma(i,k)} \; \mbox{if} \; k\in J'
\end{array}
\right.
.
\]
\end{lemma}

\begin{proof}
For the local expression of energy,
we start with the expression~\eqref{eqn:local_energy} for $E_f(H,t)$ and
split the sum over $i,j\in\{1,\ldots, m\}=J\cup J'$ into four parts
1A, 2A, 1B, 2B according to $j\in J$ or $j\in J'$ (cases with 1 or 2)
and $i\in J$ or $i \in J'$ (cases with~A or~B):
\[
\frac{1}{m} 
\left[ 
\underbrace{\left( \sum_{(i,j)\in J\times J} (\ast) \right)}_{\mbox{part 1A}} 
+
\underbrace{\left( \sum_{(i,j)\in J\times J'} (\ast)  \right)}_{\mbox{part 2A}}  
+
\underbrace{\left( \sum_{(i,j)\in J'\times J} (\ast)  \right)}_{\mbox{part 1B}}  
+
\underbrace{\left( \sum_{(i,j)\in J'\times J'} (\ast)  \right)}_{\mbox{part 2B}}  
\right]
,
\] 
where $(\ast)$ is a placeholder 
for $\displaystyle\sum_{0\neq w \in t^0_i-t^0_j+L_0}  f(\exp(H)[w+t_i-t_j])$.

\bigskip

{\bf Part 1A:} 
First we reorder terms by substituting $j$ with $\sigma(i,k)$.
Here we use that $j=\sigma(i,k)$ is a bijection of $J$ for fixed $i$, mapping index~$k$ to~$j$.
So Part~1A is equal to
\[
\sum_{i \in J} \; \sum_{k \in J} \; \sum_{0\neq w \in t^0_i-t^0_{\sigma(i,k)}+L_0}  f(\exp(H)[w+t_i-t_{\sigma(i,k)}])
.
\]
The translate $t^0_i-t^0_{\sigma(i,k)}+L_0$ can be written as $t^0_{\ell}+L_0$ with 
$\ell =\sigma(i,j)=\sigma(i,\sigma(i,k))\in J$ depending on~$i$ and~$k$. 
So we get for Part~1A:
\[
\sum_{i \in J} \; \sum_{k \in J} \; \sum_{0\neq w \in t^0_{\ell}+L_0}  f(\exp(H)[w+t_i-t_{\sigma(i,k)}])
\]
with the vectors $w \in t^0_{\ell}+L_0$ for $\ell \in J$ 
running through all non-zero elements of the lattice~$ L_{\max}$.
Therefore a shift of the $w$ by any vectors of $ L_{\max}$ does not effect the outcome 
for Part~1A. For every $j\in J$ we may shift by $-t^0_j$ and get the same value 
as for Part~1A also in
\[
\sum_{i \in J} \; \sum_{k \in J} \; \sum_{0\neq w \in t^0_{\ell} -t^0_j  +L_0}  f( \cdots )
=
\sum_{i \in J} \; \sum_{k \in J} \; \sum_{0\neq w \in t^0_{\sigma(l,j)}+L_0}  f( \cdots )
.
\]
Here and in the following $f( \cdots )$ abbreviates $f(\exp(H)[w+t_i-t_{\sigma(i,k)}])$.
Since $$\bigcup_{j\in J}  t^0_{\sigma(l,j)}+L_0 =  L_{\max}$$ for every fixed~$\ell$, we can take an average over all $j\in J$
and get for Part~1A:
\[
\frac{1}{|J|}
\sum_{i \in J} \; 
\sum_{k \in J} \; 
\sum_{0\neq w \in  L_{\max}} 
f( \cdots )
=
\frac{1}{|J|}
\sum_{0\neq w \in  L_{\max}} \; 
\sum_{i \in J} \; 
\sum_{k \in J}
 f( \cdots )
\]

\bigskip

{\bf Part 2A:}  
First we reorder terms again, by substituting $j$ with $\sigma(i,k)$.
Here we use that $j=\sigma(i,k)$ is a bijection of $J'$ for fixed $i$, mapping index~$k$ to~$j$.
So Part~2A is equal to
\[
\sum_{i \in J} \; \sum_{k \in J'} \; \sum_{0\neq w \in t^0_i-t^0_{\sigma(i,k)}+L_0}  f(\exp(H)[w+t_i-t_{\sigma(i,k)}])
.
\]
The translate $t^0_i-t^0_{\sigma(i,k)}+L_0$ can be written as $-t^0_{\ell}+L_0$ with 
$\ell =\sigma(i,j)=\sigma(i,\sigma(i,k))\in J'$ depending on~$i$ and~$k$. 
So Part~2A can be written as:
\[
\sum_{i \in J} \; \sum_{k \in J'} \; \sum_{0\neq w \in -t^0_{\ell}+L_0}  f(\exp(H)[w+t_i-t_{\sigma(i,k)}])
\]
with the vectors $w \in -t^0_{\ell}+L_0$ for $\ell \in J'$ 
running through all non-zero elements of the lattice translate~$-( v+L_{\max})$.
A shift of the $w$ by any vectors of $ L_{\max}$ does not effect the outcome 
for Part~2A. So for every $j\in J$ we may shift by $t^0_j$ and get the same value 
as for Part~2A also in
\[
\sum_{i \in J} \; \sum_{k \in J} \; \sum_{0\neq w \in t^0_j -t^0_{\ell}  +L_0}  f( \cdots )
=
\sum_{i \in J} \; \sum_{k \in J} \; \sum_{0\neq w \in -t^0_{\sigma(j,l)}+L_0}  f( \cdots )
.
\]
Here, $\sigma(j,l)\in J'$ since $j\in J$ and $\ell \in J'$,
and $f( \cdots )$ abbreviates $f(\exp(H)[w+t_i-t_{\sigma(i,k)}])$ again.
Since $\bigcup_{j\in J}  -t^0_{\sigma(j,l)}+L_0 = -( v+L_{\max})$ for every fixed~$\ell$, we can take an average over all $j\in J$
and get for Part~2A:
\[
\frac{1}{|J|}
\sum_{i \in J} \; 
\sum_{k \in J'} \; 
\sum_{0\neq w \in -( v+L_{\max})} 
f( \cdots )
=
\frac{1}{|J|}
\sum_{0\neq w \in -( v+L_{\max})} \; 
\sum_{i \in J} \; 
\sum_{k \in J'}
 f( \cdots )
\]

\bigskip

{\bf Part 1B:} 
We start by substituting $j$ with $\sigma(i,k)$ again, where
$j=\sigma(i,k)$ is a bijection from $J'$ to $J$ for fixed $i$, mapping index~$k$ to~$j$.
So Part~1B is equal to
\[
\sum_{i \in J'} \; \sum_{k \in J'} \; \sum_{0\neq w \in t^0_i-t^0_{\sigma(i,k)}+L_0}  f(\exp(H)[w+t_i-t_{\sigma(i,k)}])
.
\]
The translate $t^0_i-t^0_{\sigma(i,k)}+L_0$ can be written as $t^0_{\ell}+L_0$ with 
$\ell =\sigma(i,j)=\sigma(i,\sigma(i,k))\in J'$ depending on~$i$ and~$k$. 
So Part~1B can be written as:
\[
\sum_{i \in J'} \; \sum_{k \in J'} \; \sum_{0\neq w \in t^0_{\ell}+L_0}  f(\exp(H)[w+t_i-t_{\sigma(i,k)}])
\]
with the vectors $w \in t^0_{\ell}+L_0$ for $\ell \in J'$ 
running through all non-zero elements of the lattice translate~$ v+L_{\max}$.
Again, a shift of the $w$ by any vectors of $ L_{\max}$ does not effect the outcome 
for Part~1B. So for every $j\in J$ we may shift by $-t^0_j$ and get the same value 
as for Part~1B also in
\[
\sum_{i \in J'} \; \sum_{k \in J'} \; \sum_{0\neq w \in t^0_{\ell} -t^0_j  +L_0}  f( \cdots )
=
\sum_{i \in J'} \; \sum_{k \in J'} \; \sum_{0\neq w \in t^0_{\sigma(l,j)}+L_0}  f( \cdots )
.
\]
Here, $\sigma(l,j)\in J'$ since $\ell \in J'$ and $j\in J$,
and $f( \cdots )$ abbreviates $f(\exp(H)[w+t_i-t_{\sigma(i,k)}])$ again.
Since $\bigcup_{j\in J}  t^0_{\sigma(l,j)}+L_0 = ( v+L_{\max})$ for every fixed~$\ell$, we can take an average over all $j\in J$
and get for Part~1B:
\[
\frac{1}{|J|}
\sum_{i \in J'} \; 
\sum_{k \in J'} \; 
\sum_{0\neq w \in ( v+L_{\max})} 
f( \cdots )
=
\frac{1}{|J|}
\sum_{0\neq w \in ( v+L_{\max})} \; 
\sum_{i \in J'} \; 
\sum_{k \in J'}
 f( \cdots )
\]

\bigskip

{\bf Part 2B:} 
We reorder terms by substituting $j$ with $\sigma(i,k)$
where $j=\sigma(i,k)$ is a bijection from $J$ to $J'$ for fixed $i$, 
mapping index~$k$ to~$j$.
So Part~2B is equal to
\[
\sum_{i \in J'} \; \sum_{k \in J} \; \sum_{0\neq w \in t^0_i-t^0_{\sigma(i,k)}+L_0}  f(\exp(H)[w+t_i-t_{\sigma(i,k)}])
.
\]
The translate $t^0_i-t^0_{\sigma(i,k)}+L_0$ can be written as $-t^0_{\ell}+L_0$ with 
$\ell =\sigma(i,j)=\sigma(i,\sigma(i,k))\in J$ depending on~$i$ and~$k$. 
So we get for Part~2B:
\[
\sum_{i \in J'} \; \sum_{k \in J} \; \sum_{0\neq w \in -t^0_{\ell}+L_0}  f(\exp(H)[w+t_i-t_{\sigma(i,k)}])
\]
with the vectors $w \in -t^0_{\ell}+L_0$ for $\ell \in J$ 
running through all non-zero elements of the lattice~$ L_{\max}=- L_{\max}$.
A shift of the $w$ by any vectors of $ L_{\max}$ does not effect the outcome 
for Part~2B. In particular, for every $j\in J$ we may shift by $t^0_j$ and get the same value 
as for Part~2B also in
\[
\sum_{i \in J'} \; \sum_{k \in J} \; \sum_{0\neq w \in t^0_j -t^0_{\ell}  +L_0}  f( \cdots )
=
\sum_{i \in J'} \; \sum_{k \in J} \; \sum_{0\neq w \in t^0_{\sigma(j,l)}+L_0}  f( \cdots )
.
\]
Here, $\sigma(j,l)\in J$ since $j,\ell \in J$ 
and $f( \cdots )$ abbreviates $f(\exp(H)[w+t_i-t_{\sigma(i,k)}])$.
Since $\bigcup_{j\in J}  t^0_{\sigma(j,l)}+L_0 =  L_{\max}$ for every fixed~$\ell$, we can take an average over all $j\in J$
and get for Part~2B:
\[
\frac{1}{|J|}
\sum_{i \in J'} \; 
\sum_{k \in J} \; 
\sum_{0\neq w \in  L_{\max}} 
f( \cdots )
=
\frac{1}{|J|}
\sum_{0\neq w \in  L_{\max}} \; 
\sum_{i \in J'} \; 
\sum_{k \in J}
 f( \cdots )
\]

\bigskip

{\bf Summing all up:} 
Finally, we can combine Parts~1A and~2B to get:
\[
\frac{1}{|J|} 
 \sum_{0\not= w \in  L_{\max}} \; \sum_{i \in  J\cup J'} \; \sum_{k\in J} \;  
     f\left( \exp(H) [ w+t_i-t_{\sigma(i,k)}] \right)
\]
Altogether, with Parts 2A and 1B and with the observation $|J|=\frac{m}{2}$, 
we get the asserted formula for $E_f(H,t)$.
\end{proof}

\section{The $\mathsf{D}_n^+$ example}\label{sec:four}


For $n\geq 1$ the lattice $\mathsf{D}_n$ consists of all 
integral vectors with an even coordinate sum:
$$
\mathsf{D}_n = \left\{ x\in \Z^n \; : \; x_1 + \ldots + x_n \; \mbox{ even } \right\}
$$
The lattice is sometimes also referred to as the {\em checkerboard lattice}.
It gives one of the two families of irreducible {\em root lattices} which exist in every dimension, the other one being~$\mathsf{A}_n$.

The set $\mathsf{D}^+_n$ is defined as the $2$-periodic set
$$
\mathsf{D}^+_n \; = \; \mathsf{D}_n \cup \left( \frac{\mathbb{1}}{2} + \mathsf{D}_n \right)
,
$$
where $\mathbb{1}$ stands for the all-one vector $(1,\ldots , 1)^t $
It is easy to show that $\mathsf{D}^+_n$ is a lattice if and only if $n$ is even,
as the vector $2 \frac{\mathbb{1}}{2}=\mathbb{1}$ is an element of $\mathsf{D}_n$ only 
if $n$ is even.


For $n=8$, $\mathsf{D}^+_n$ is equal to the famous root lattice
$\mathsf{E}_8$, with a lot of remarkable properties, not only for
energy minimization (see \eg \cite{cs1999}).
For $n=9$, $\mathsf{D}^+_n$ is a  $2$-periodic non-lattice set
sharing several of the remarkable properties of~$\mathsf{E}_8$.
It is for instance also a conjectured optimal sphere packing in its dimension,
although as such it is not unique, but part of an infinite family of 
``fluid diamond packings'' in dimension~$9$.
Besides its putative optimality for the more general energy
minimization problem (see \cite{PhysRev09}),
$\mathsf{D}^+_9$ has for instance also been found to give the best known set 
for the quantization problem, being in particular better than any
lattice in dimension~$9$ (see \cite{ae1998}).


In the following we collect some of the properties of
$\mathsf{D}^+_n$, which are needed in later sections.
We start with its symmetries.


The finite orthogonal group preserving $\mathsf{D}_n$ contains 
the {\em hyperoctahedral group},
which is isomorphic to $S_n\rtimes (\Z/2\Z)^{n}$,
since every coordinate permutation and every sign flip leaves the parity of 
the coordinate sum unchanged.
Only for $n=4$ there exists an additional threefold symmetry 
(see \eg \cite[Section 4.3]{MR1957723}).

The group $\Aut_0 \mathsf{D}^+_n= \left\lbrace \varphi  \in \Aut \mathsf{D}_n \mid \varphi(\mathsf{D}^+_n)=\mathsf{D}^+_n \right\rbrace$, contains all the 
coordinate permutations and every even number of sign flips,
so it is a group isomorphic to $S_n\rtimes (\Z/2\Z)^{n-1}$.
This is precisely the {\em Weyl group} $W(\mathsf{D}_n)$ of the $\mathsf{D}_n$ {\em root system}
(the minimal vectors of $\mathsf{D}_n$). 

For even $n$, this gives all automorphisms of the lattice
$\mathsf{D}^+_n$ 
(see \lc), \ie we have $\Aut \mathsf{D}^+_n=\Aut_0 \mathsf{D}^+_n=W(\mathsf{D}_n)$. 

For odd $n$, the maximal period lattice of $\mathsf{D}^+_n$ is $\mathsf{D}_n$, and it
follows from the discussion in Remark~\ref{rem:auto-discuss} that $\Aut_0
\mathsf{D}^+_n=W(\mathsf{D}_n)$ has index $2$ in $\Aut \mathsf{D}^+_n=W(\mathsf{D}_n) \cup -W(\mathsf{D}_n)$. The
orthogonal automorphisms of $\mathsf{D}^+_n$ coming from $W(\mathsf{D}_n)$ correspond to
affine isometries fixing $0$ and $\frac{\mathbb{1}}{2}$ 
modulo $L_{\max}=\mathsf{D}_n$ while those from $ -W(\mathsf{D}_n)$ correspond to affine isometries exchanging $0$ and $\frac{\mathbb{1}}{2}$. In particular, all non-empty shells $\Lambda_x(r)$ of $\mathsf{D}_n^+$ are fixed by $W(\mathsf{D}_n)$. We will take advantage of this invariance property in the sequel, using classical results about the invariant theory of the Weyl group $W(\mathsf{D}_n)$.

\begin{proposition}
Every non-empty shell $\Lambda(r)$ and $\Lambda_x(r)$ of $\Lambda=\mathsf{D}_n^+$ 
forms a spherical $3$-design.
\end{proposition}

\begin{proof}
For a finite set $X$ on a sphere of radius~$r$ being a spherical $3$-design is 
equivalent to
$$
\sum_{x\in X} ( y^t x )^2 = c \cdot ( y^t y )
\quad \mbox {and} \quad
\sum_{x\in X} ( y^t x )^3 = 0
$$
for some constant $c$ and any $y \in \R^n$.
The first property is actually that of a $2$-design. 
It is satisfied for any set $X$ which is invariant under a group that
acts irreducibly on $\R^n$ (see \cite[Theorem 3.6.6.]{MR1957723} where the synonymous expression "strongly eutactic configuration" is used ) . 

The second property is satisfied, since the Weyl group of the root system~$\mathsf{D}_n$ 
has no non-zero invariant homogeneous polynomials of degree~$3$
(see \cite[\S 3.7,  Table 1]{MR1066460}).
\end{proof}
\begin{rem}
 For half-integral $r$ the shells $\Lambda_x(r)$ are not centrally symmetric  and therefore the $2$-design property does not immediately imply the $3$-design property.
\end{rem}

As a consequence of the preceding proposition, $\mathsf{D}_n^+$
satisfies the properties of Theorem~\ref{th:periodic-critical}: the shells 
$\Lambda_x(r)$ are balanced and $\Lambda(r)$ is a spherical $2$-design for all~$r$. 
Consequently, $\mathsf{D}_n^+$ is $f_c$-critical for any~$c>0$. On the other hand, the shells 
are not $4$-designs in general, as can be checked numerically for small $r$. If they 
were, then the study of the Hessian in the following section
would be significantly simpler, in the spirit of what 
was done in~\cite{MR2889159}.

\section{The Hessian of $2$-periodic sets 
              and in particular of $\mathsf{D}_n^+$} \label{sec:hessian}

For $f(r)=e^{-cr}$, we consider the Hessian of $E_{f_c}(H,\mathbf{t})$ at a $2$-periodic set $L_{\max}\cup(v+L_{\max})$ given by an 
$m$-periodic representation $\bigcup_{i=1}^m (t^0_i + L_0)$. We will then use the obtained expression for the Hessian to analyze whether or not $\mathsf{D}_n^+$ is a local minimum among $m$-periodic sets.

According to Section \ref{sec:two} this Hessian is equal to

\begin{equation}\label{eq:splithess} 
\frac{c}{m} 
\sum_{r>0} \left[I(r)+II(r)+III(r)\right] e^{-c r^2}
\end{equation}

where
\begin{align*}
I(r)&=\sum_{1\leq i,j\leq m}\;\sum_{w \in t^0_i-t^0_j+L_0, \|w \|=r}2c \left( w^{t}(t_i-t_j)\right)^2  -  \| t_i-t_j\|^2\\
II(r)&=\sum_{1\leq i,j\leq m}\;\sum_{w \in t^0_i-t^0_j+L_0, \|w \|=r}- 2w^{t}H(t_i-t_j)+2cw^{t}(t_i-t_j)H[w]\\
III(r)&=\sum_{1\leq i,j\leq m}\;\sum_{w \in t^0_i-t^0_j+L_0, \|w
        \|=r}\frac{c}{2}H[w]^2-\frac{1}{2}H^2\br{w} 
\end{align*}

In this decomposition we distinguish three types of terms: purely translational terms ($I(r)$), mixed terms ($II(r)$) and purely lattice changing terms ($III(r)$). Note that we can reorder individually each of these three terms according to Lemma \ref{lem:ReorderingLemma}. In particular we will use that
\begin{eqnarray*}
I(r)
& = &
\frac{2}{m}\left[
\left(
 \sum_{\substack{0\not= w \in  L_{\max}\\ \|w\| = r}} \; \sum_{i=1}^m \; \sum_{k\in J} \;  
     2c \left( w^{t}(t_i-t_{\sigma(i,k)})\right)^2  -  \| t_i-t_{\sigma(i,k)}\|^2
\right)
\right. \\
& & \qquad +
\left(
\sum_{\substack{0\not= w \in -(v+ L_{\max})\\ \|w\| = r}} \; \sum_{i\in J} \; \sum_{k\in J'} \;   
     2c \left( w^{t}(t_i-t_{\sigma(i,k)})\right)^2  -  \| t_i-t_{\sigma(i,k)}\|^2
\right) \\
& & \qquad +
\left.
\left(
\sum_{\substack{0\not= w \in (v+ L_{\max})\\ \|w\| = r}} \; \sum_{i\in J'} \; \sum_{k\in J'} \;   
     2c \left( w^{t}(t_i-t_{\sigma(i,k)})\right)^2  -  \| t_i-t_{\sigma(i,k)}\|^2
\right) 
\right],
\end{eqnarray*}
where we assume that $t^0_i\in L_{\max}$ for $i\in J=\{1,\ldots,\frac{m}{2}\}$, and $t^0_i \in v + L_{\max}$ for $i\in J'=\{\frac{m}{2}+1,\ldots, m\}$, which finally simplifies to
\begin{equation}
\begin{aligned}
I(r)
 &= 
\frac{2}{m}\left[
\left(
 \sum_{\substack{0\not= w \in  L_{\max}\\ \|w\| = r}} \; \sum_{i=1}^m \; \sum_{k\in J} \;  
     2c \left( w^{t}(t_i-t_{\sigma(i,k)})\right)^2  -  \| t_i-t_{\sigma(i,k)}\|^2
\right)
\right. \\
 &\qquad +
\left.
\left(
\sum_{\substack{0\not= w \in (v+ L_{\max})\\ \|w\| = r}} \; \sum_{i=1}^m \; \sum_{k\in J'} \;   
     2c \left( w^{t}(t_i-t_{\sigma(i,k)})\right)^2  -  \| t_i-t_{\sigma(i,k)}\|^2
\right) 
\right]
\end{aligned}
\end{equation}
since the inner sums are invariant towards negation of~$w$.

\subsection{Purely translational terms for $\mathsf{D}_n^+$}\label{ptt}

This formula simplifies for $\Lambda=\mathsf{D}^+_n$ with odd $n$ since the elements of a given non-²empty shell $\Lambda(r)$ are either all contained in $\mathsf{D}_n$
or in $\pm \frac{\mathbb{1}}{2} + \mathsf{D}_n$, depending on wether $r$ is integral or half-integral.
This gives us two cases to consider:

In one case, assuming $\Lambda(r)\subset \mathsf{D}_n$ we get, for fixed $r>0$,
\begin{equation} \label{eqn:CaseI_TranslationalPart}
I(r)=\frac{2}{m}\sum_{w\in \mathsf{D}_n, \|w\|=r} \;
\sum_{i=1}^m \;
\sum_{k=1}^{m/2} \;
2c
\left( w^{t}(t_i-t_{\sigma{(i,k)}})\right)^2
-  \| t_i-t_{\sigma{(i,k)}}\|^2
\end{equation}
and in the other case,

\begin{equation} \label{eqn:CaseI_TranslationalPartbis}
I(r)=\frac{2}{m}\sum_{w\in \left(\frac{\mathbb{1}}{2} + \mathsf{D}_n\right), \|w\|=r} \;
\sum_{i=1}^m \;
\sum_{k=m/2+1}^{m} \;
2c
\left( w^{t}(t_i-t_{\sigma{(i,k)}})\right)^2
-  \| t_i-t_{\sigma{(i,k)}}\|^2
\end{equation}

In both cases, we can use the relation
\begin{align}\label{eqn:TraceRewriteTranslationalTerms}
\left( w^{t}(t_i-t_{\sigma(i,k)})\right)^2 
& =  w^t (t_i-t_{\sigma(i,k)}) (t_i-t_{\sigma(i,k)})^t w\\
\notag & =\Tr \left( (t_i-t_{\sigma(i,k)}) (t_i-t_{\sigma(i,k)})^t (w w^t) \right).
\end{align}
to simplify the part of the sum involving $w$.

Using the linearity of the trace we get in the first case, that is if $\Lambda(r)\subset \mathsf{D}_n$,
\begin{equation}\label{tr1}
\begin{split}
&\sum_{w\in  \mathsf{D}_n, \|w\|=r} \;\sum_{i=1}^{m} \;
\sum_{k=1}^{m/2} \;
2c
\left( w^{t}(t_i-t_{\sigma{(i,k)}})\right)^2\\
=&\quad 2c
\Tr
\left(
\left(
\sum_{i=1}^m \;
\sum_{k=1}^{m/2} \; 
(t_i-t_{\sigma{(i,k)}}) (t_i-t_{\sigma{(i,k)}})^t 
\right)
\sum_{w\in \mathsf{D}_n, \|w\|=r} \;
(w w^t)
\right) 
\end{split}
\end{equation}
and in the second case, 
\begin{equation}\label{tr2} 
\begin{split}
&\sum_{w\in  \left( \frac{\mathbb{1}}{2}  + \mathsf{D}_n\right), \|w\|=r} \;\sum_{i=1}^{m} \;
\sum_{k=m/2+1}^{m} \;
2c
\left( w^{t}(t_i-t_{\sigma{(i,k)}})\right)^2\\
=&\quad 2c
\Tr
\left(
\left(
\sum_{i=1}^m \;
\sum_{k=m/2+1}^{m} \; 
(t_i-t_{\sigma{(i,k)}}) (t_i-t_{\sigma{(i,k)}})^t 
\right)
\sum_{w\in  \left( \frac{\mathbb{1}}{2}  + \mathsf{D}_n\right), \|w\|=r} \;
(w w^t)
\right).
\end{split}
\end{equation}
Using the $2$-design property of the shell $\Lambda(r)$
(see \eqref{2des}), and noticing that a typical element $w$ of $\mathsf{D}^+_n$ has weight
$$\nu(w)=\begin{cases}1 \text{ if } w \in \mathsf{D}_n\\ \frac{1}{2} \text{ if }w \in \frac{\mathbb{1}}{2}  + \mathsf{D}_n \end{cases}$$
 we may substitute
$\displaystyle\sum_{w\in \mathsf{D}_n, \|w\|=r} \; (w w^t)$
and $\displaystyle\sum_{w\in  \left( \frac{\mathbb{1}}{2}  + \mathsf{D}_n\right), \|w\|=r} \;
(w w^t)$ by
$\frac{r^2|\Lambda(r)|}{n} \id_n$. Therefore, formula \eqref{tr1} and \eqref{tr2}  simplify respectively to 
$$
\frac{2c r^2|\Lambda(r)|}{n} 
\sum_{i=1}^m \;
\sum_{k=1}^{m/2} \; 
\| t_i-t_{\sigma{(i,k)}}) \|^2
 \text{ and } 
 \frac{2c r^2|\Lambda(r)|}{n} 
\sum_{i=1}^m \;
\sum_{k=m/2+1}^{m} \;
\| t_i-t_{\sigma{(i,k)}} \|^2.
$$
We finally get
 
$$
I(r)=\frac{2}{m}\left( \frac{2c r^2}{n}-1 \right) |\Lambda(r)| 
\sum_{i=1}^m \;
\sum_{k=1}^{m/2} \; 
\| t_i-t_{\sigma{(i,k)}} \|^2
$$
in the first case ($\Lambda(r)\subset \mathsf{D}_n$) and
$$
I(r)=\frac{2}{m}\left( \frac{2c r^2}{n}-1 \right) |\Lambda(r)|
\sum_{i=1}^m \;
\sum_{k=m/2+1}^{m} \; 
\| t_i-t_{\sigma{(i,k)}} \|^2
$$
in the second case.
In both cases, this is nonnegative for all $c \geq \frac{n}{2r^2}$.

As $r^2\geq 2$ for $\mathsf{D}^+_n$ with $n\geq 8$, we overall find
for $\Lambda=\mathsf{D}^+_n$ 
that the purely translational terms are nonnegative
for all $c \geq \frac{n}{4}$ and $n\geq 8$.

\subsection{Mixed terms}\label{mt}

There are two different mixed terms in $II(r)$:
The first one is the sum over terms $w^{t}H(t_i-t_j)$
and the second one is the sum over terms $w^{t}(t_i-t_j)H[w]$.

The first sum evaluates to~$0$ for balanced configurations
as it can be reordered as follows:
$$
\sum_{i,j}\sum_{w \in t^0_i-t^0_j+L_0 ,  \|w\|=r}  w^{t}H(t_i-t_j)
=
\Tr \left( H \cdot ( \sum_{i,j}\sum_{w \in t^0_i-t^0_j+L_0 ,  \|w\|=r}  (t_i-t_j)w^t ) \right) 
$$
with
$$
\sum_{i,j}\sum_{w \in t^0_i-t^0_j+L_0 ,  \|w\|=r}  (t_i-t_j)w^t 
\;
=
\;
\sum_{1\leq k\leq m} 
t_k 
\left(
\; \sum_{u \in t^0_k-t^0_{\ell}+L_0 \, \atop \mbox{\footnotesize for some } \, 1\leq l \leq m , \; \|u\| = r}  \;
u
\right)
$$
as seen in the proof of Theorem~\ref{th:periodic-critical}.
Thus for balanced shells $\Lambda(r)$ this part of the Hessian vanishes.

\bigskip

For the second sum of mixed terms over a fixed shell we get:
\begin{eqnarray*}
& & 
\sum_{i,j}\sum_{w \in t^0_i-t^0_j+L_0 ,  \|w\|=r}
w^{t}(t_i-t_j)H[w] 
\\
&  =  & 
\sum_{i,j}\sum_{w \in t^0_i-t^0_j+L_0 ,  \|w\|=r}  
H[w] w^{t} t_i 
-
\sum_{i,j}\sum_{w \in t^0_i-t^0_j+L_0 ,  \|w\|=r}  
H[w] w^{t} t_j
\\
&  =  & 
\sum_{i} 
\left( \sum_j \sum_{w \in t^0_i-t^0_j+L_0 ,  \|w\|=r}   H[w] w^{t} \right) t_i 
-
\sum_{j} 
\left( \sum_i \sum_{w \in t^0_i-t^0_j+L_0 ,  \|w\|=r}   H[w] w^{t} \right) t_j 
\\
&  =  & 
\sum_{i} 
\left( \sum_{w\in -\Lambda_{t_i^0}, \|w\|=r}  H[w] w^{t} \right) t_i 
-
\sum_{j} 
\left( \sum_{w\in \Lambda_{t_j^0}, \|w\|=r}  H[w] w^{t} \right) t_j  
\\
& = & 
 -2 \sum_{i} 
\left( \sum_{w\in \Lambda_{t_i^0}, \|w\|=r}  H[w] w^{t} \right) t_i 
\end{eqnarray*}
Here the inner sum is a homogeneous degree~$3$ polynomial in $w$ evaluated
on the shell $\Lambda_{t_i^0}(r)$. Since these shells are $3$-designs 
for $\Lambda=\mathsf{D}_n^+$, the inner sum vanishes for all shells of
$\mathsf{D}_n^+$. Indeed, any degree $3$ homogeneous polynomial $P(w)$
decomposes uniquely as a sum $P(w)=F(w)+\|w\|^2 G(w)$ where $F(w)$ is
a harmonic degree $3$ polynomial and $G(w)$ is a linear
form. Consequently the sum $\sum_{w \in X} P(w)$, where $X$ is any
spherical $3$-design contained in a sphere of radius~$r$, reduces to
$$\sum_{w \in X} P(w)=\sum_{w \in X} F(w)+r^2\sum_{w \in X} G(w)$$
and both the sums $\sum_{w \in X} F(w)$ and $\sum_{w \in X} G(w)$
vanish from the $3$-design property.

\subsection{Purely lattice changing terms in the case of $\mathsf{D}_n^+$}
\label{sec:plc}

It remains to look at the sum $III(r)$, which we can also write as

\begin{equation}\label{eq:h1}
III(r)=m\sum_{w \in \Lambda_r} \nu(w)\left(
  \frac{c}{2}H[w]^2-\frac{1}{2}H^2\br{w} \right)
.
\end{equation}
This sum corresponds to an effect coming from local changes of 
the underlying lattice~$L_0$, respectively of $\mathsf{D}_n$ in case of
$\mathsf{D}_n^+$.

The sum of the terms $H^2[w]$ over any given shell simplifies to
\begin{equation}\label{eq:h2}
m\sum_{w\in \Lambda(r)}\nu(w)H^2\br{w}
\; = \; m\,\nu_r\dfrac{r^2 |\Lambda(r)|}{n}\Tr H^2
\end{equation}
because of the weighted-$2$-design property,  where
$\nu_r=\dfrac{\sum_{w \in \Lambda(r)} \nu(w)}{|\Lambda(r)|}$ is the
average weight on~$\Lambda(r)$. In the case of $\mathsf{D}_n^+$, the
weight is constant ($1$ or $\frac{1}{2}$) on each $\Lambda(r)$, so that \eqref{eq:h1} simplifies to
	\begin{equation}\label{eq:h3}
	m \, \nu_r \sum_{w \in \Lambda(r)} \left(
          \frac{c}{2}H[w]^2-\frac{1}{2}H^2\br{w} \right)
   \; = \; m \, \nu_r
        \sum_{w \in \Lambda(r)} \left( \frac{c}{2}H[w]^2\right)
      -m \, \nu_r\dfrac{r^2 |\Lambda(r)|}{n}\Tr H^2.
	\end{equation}

For the terms involving $(H[w])^2$, we note that, for any positive $r$, the 
polynomial $\sum_{w \in \Lambda(r)}H[w]^2$ 
is a quadratic $G$-invariant polynomial 
in $H$, where $G=\Aut_0(\mathsf{D}_n^+)=W(\mathsf{D}_n)$. We will make use of  the following 
classical result  about the polynomial invariants of $G$.
\begin{lemma} \label{lem:quadratic_poly_on_sd}
Let $n\geq 5$.
Then any homogeneous quadratic polynomial 
on the space $\Sn$ of symmetric $n\times n$ matrices $H=(h_{ij})$,
which is invariant under the Weyl group 
of $\mathsf{D}_n$
(acting on $\Sn$ by $H\mapsto M^t H M$ by the $n\times n$ permutation matrices and
diagonal matrices~$M$ having an even number of $-1$s and $1$s otherwise on the diagonal),
is a linear combination of the
three quadratic polynomials 
\[
(\Tr H)^2 =  \sum_{i,j=1}^{n}h_{jj}h_{ii}, \quad
\Tr (H^2) = \sum_{i,j=1}^{n}h^2_{ij} \quad
\mbox{and} \quad
\sum_{i<j} h^2_{ij}
.
\]
\end{lemma}

\begin{proof}
Since we are not aware of a pinpoint reference for this statement, we give a short argument here for the convenience of the reader. The homogeneous quadratic polynomials on $\Sn$  
have seven types of monomials
(where different indices are actually chosen to be different):
\[
h_{ii}^2, \quad
h_{ii} h_{jj},  \quad
h_{ii} h_{ij}, \quad  
h_{ii} h_{jk}, \quad 
h_{ij}^2, \quad  
h_{ij} h_{ik}, \quad  
h_{ij} h_{kl}
\]
Note that there are less of these monomials for $n=2,3$.

From the invariance towards permutation matrices we can conclude 
that coefficients in front of any given type of monomials have to be the same. 
From the invariance towards 
diagonal matrices with an even number of $-1$s (and $1$s otherwise)
we then deduce that only monomials of the three types 
$h_{ii}^2$, $h_{ii}h_{jj}$ and $h_{ij}^2$ are invariant under the Weyl group of $\mathsf{D}_n$.
Among the others, some monomials are mapped to their negatives.
The only exception is the case $n=4$, where also the set of monomials of the type
$h_{ij} h_{kl}$ is invariant under the action of the group.
\end{proof}

Note that the lemma and its proof can be adapted to the description of the space of quadratic $G$-invariant differential operators on functions with matrix argument. In particular, this space has dimension $3$. Using the local system of coordinates $h_{ij}, \, 1 \leq i \leq j \leq n$, of $\Sn$ and denoting by $\partial_{ij}$ the partial derivative with respect to $h_{ij}$, a spanning system is given by
\begin{align*}
\delta_{1}&=\dfrac{1}{n(n-1)}\sum_{i<j}\partial_{ii}\partial_{jj},\\
\delta_{2}&=\dfrac{1}{2n}\sum_{i}\partial^2_{ii}-\dfrac{1}{n(n-1)}\sum_{i<j}\partial_{ii}\partial_{jj}\quad\text{ and } \\
\delta_{3}&=-\dfrac{1}{n}\sum_{i}\partial^2_{ii}+\dfrac{2}{n(n-1)}
\sum_{i<j}\partial_{ii}\partial_{jj}+\dfrac{1}{n(n-1)}\sum_{i<j}\partial_{ij}^2.
\end{align*}

This particular basis satisfies the relations $$\delta_i(F_j)=\delta_{ij},\quad 1 \leq i,j \leq 3$$  for $$F_1(H)=(\Tr 
H )^2=\left(\sum_{i=1}^n h_{ii}\right)^2,$$ $$F_2(H)=\Tr H^2=\sum_{i=1}^n h_{ii}^2+2\displaystyle\sum_{1\leq i<j\leq n}h_{ij}^2$$ and $$F_3(H)=\displaystyle\sum_{1 \leq i<j\leq n}  h_{ij}^2.$$

For any positive $r$, the polynomial $\sum_{w \in \Lambda(r)}H[w]^2$
 is a quadratic $G$-invariant polynomial in $H$. As such, it is a linear combination 
\begin{equation}\label{abc} 
\sum_{w \in \Lambda(r)}H[w]^2=\alpha_r F_1(H)
+\beta_r F_2(H) +\gamma_r F_3(H)
\end{equation}
for some constants $\alpha_r$, $\beta_r$ and $\gamma_r$ to be computed. 
To compute the constants $\alpha_r$, $\beta_r$ and $\gamma_r$ in \eqref{abc}, it suffices to evaluate  $\delta_i(\sum_{w \in \Lambda(r)}H[w]^2)$ for $1\leq i \leq 3$: setting $Z_r =\sum_{w \in \Lambda(r)}\left(\sum_{i=1}^n w_i^4\right)$, one has

\begin{align*}
\alpha _r &=\delta_1(\sum_{w \in \Lambda(r)}H[w]^2)=\dfrac{1}{n(n-1)}\left(r^4\vert \Lambda(r)\vert - Z_r\right)\\
\beta_r &=\delta_2(\sum_{w \in \Lambda(r)}H[w]^2)=\dfrac{1}{n-1}Z_r -\dfrac{1}{n(n-1)}r^4\vert\Lambda(r)\vert\\
\gamma_r &=\delta_3(\sum_{w \in \Lambda(r)}H[w]^2)=-2\dfrac{n+2}{n(n-1)}Z_r+\dfrac{6}{n(n-1)}r^4\vert\Lambda(r)\vert.
\end{align*}
\bigskip

We are now in the position to estimate $III(r)$.
Recall that we restrict to $H$ with $\Tr H =0$, in which case $F_1(H)=0$. Using the above formulas, the relation $\Tr H^2= \sum_i h_{ii}^2 + 2\sum_{i<j}h_{ij}^2$ and formula \eqref{eq:h2}, we get
\begin{align*}
\dfrac{1}{m \, \nu_r}III(r)&=\left(c\beta_r+\dfrac{c\gamma_r}{2}-\dfrac{r^2}{n}\vert\Lambda(r)\vert\right)\sum_{i<j}h_{ij}^2 + \dfrac{1}{2}\left(c\beta_r-\dfrac{r^2}{n}\vert\Lambda(r)\vert\right)\sum_i h_{ii}^2\\
&=\left(\dfrac{2c}{n(n-1)}\left(r^4\vert \Lambda(r)\vert - Z_r\right)-\dfrac{r^2}{n}\vert\Lambda(r)\vert\right)\sum_{i<j}h_{ij}^2\\& \qquad  + \dfrac{1}{2}\left(\dfrac{c}{n-1}\left(Z_r-\dfrac{r^4}{n}\vert \Lambda(r)\vert\right)-\dfrac{r^2}{n}\vert\Lambda(r)\vert\right) \sum_i h_{ii}^2.
\end{align*}
 
In order that $III(r)$ be positive, it is enough that the coefficients of $ 
\sum_{i<j}h_{ij}^2 $ and $\sum_i h_{ii}^2 $ are positive. This is of course impossible 
for small $c$, but as we show below, it is achievable for big enough $c$. To see this, we 
introduce the polynomial $$P(x)=\sum_{i=1}^n x_i^4 
-\dfrac{3}{n+2}\left(\sum_{i=1}^n x_i^2\right)^2$$ which is readily seen to be 
harmonic. As a consequence of Proposition \ref{mod}, the average theta series $f(\tau)=\theta_{\Lambda,P}(\tau)$ is a 
cusp modular form of weight $k=\dfrac{n}{2}+4$, and its Fourier coefficients $a_r(f)$ are "small", in a sense to be 
made more precise. 
Finally, from the relation 
$Z_r=a_r(f)+\dfrac{3}{n+2}r^4\vert \Lambda(r)\vert$, we can rewrite the coefficients of
$ \sum_{i<j}h_{ij}^2$ and $\sum_i h_{ii}^2$ in the expression for 
$\dfrac{1}{m\, \nu_r}III(r)$ as

\begin{equation}\label{coeff1}
\dfrac{2c}{n(n+2)}r^4\vert\Lambda(r)\vert\left[\left( 
1-\dfrac{n+2}{2cr^2}\right)-\dfrac{n+2}{n-1}\,\dfrac{a_r(f)}{r^4\vert\Lambda(r)\vert}\right]
\end{equation}
and 
\begin{equation}\label{coeff2}
\dfrac{c}{n(n+2)}r^4\vert\Lambda(r)\vert\left[\left( 
1-\dfrac{n+2}{2cr^2}\right)+\dfrac{n(n+2)}{2(n-1)}\,\dfrac{a_r(f)}{r^4\vert\Lambda(r)\vert}\right].
\end{equation}
Note that if all shells $\Lambda(r)$ were spherical $4$-designs, then
$a_r(f)$ would be zero, and the above coefficients would be positive
for any $c>\dfrac{n+2}{4}$. As mentioned before, not all shells of
$\mathsf{D}_n^{+}$ do have the $4$-design property. We can nevertheless 
obtain the same conclusion, using some classical estimates on the
growth of the coefficients of cusp forms:

\begin{lemma}\label{cusp}
	For any $r>0$ such that the shell $\Lambda(r)$ of $\mathsf{D}_n^+$ is non-empty one has
	$$\dfrac{a_r(f)}{r^4\vert\Lambda(r)\vert}=\mathcal{O}\left(r^{-\frac{n}{2}+2}\right).$$
\end{lemma}
\begin{proof}
	Using elementary bounds on the size of coefficients of cusp forms (see \eg 
	\cite[(5.7)]{MR1474964}) we see that 
	$$a_r(f)=\mathcal{O}\left(r^{\frac{n}{2}+4}\right).$$
	As for the size of $\Lambda(r)$ we can use classical estimates on the number of 
	representations by quadratic forms (see \eg \cite[chapter 11 ]{MR1474964}) . For 
	shells $\Lambda(r)$ which are contained in $\mathsf{D}_n^{+}$, corresponding to $r$ such 
	that $r^2$ is integral, one can apply Corollary 11.3 of \cite{MR1474964}  to 
	conclude that $\vert\Lambda(r)\vert \asymp r^{n-2}$. 
         For shells contained in 
        $(-\frac{\mathbb{1}}{2} + \mathsf{D}_n) \cup
        (\frac{\mathbb{1}}{2} + \mathsf{D}_n)$ 
        the same argument 
	applies since these shells are indeed shells of the lattice $\mathsf{D}_n^{*}$. 
	Altogether, we obtain the desired estimate for the quotient 
	$\dfrac{a_r(f)}{r^4\vert\Lambda(r)\vert}$.
\end{proof}

\section{Concluding remarks}  \label{sec:conclusion}

\begin{theorem} \label{thm:final}
	Let $n$ be an odd integer $\geq 9$. Then there exists a constant $c_n$ such that $\mathsf{D}_n^{+}$ is locally $f_c$-optimal for any $c > c_n$.
\end{theorem}
\begin{proof}
	This is mainly the collection of facts proven before: we know from Section~\ref{ptt} 
	that the purely translational part of the hessian is $>0$ as
        soon as $c >\frac{n}{4}$ 
and that the mixed terms vanish (Section~\ref{mt}). As for the pure lattice changes, the sign of their contribution is governed by that of \eqref{coeff1} and \eqref{coeff2}, which is positive  if $c$ 
	is big enough, thanks to  Lemma \ref{cusp}.
	\end{proof}

For $n=9$, the result of Theorem~\ref{thm:final} is of course not fully
satisfactory as one would expect local $f_c$-optimality to hold for
any $c>0$, in accordance with the conjecture and experimental results
about $\mathsf{D}_9^{+}$ mentioned at the beginning of this paper. A strategy
to get such a universal local optimality result --- which we used in
\cite{MR2889159} for the lattices $\mathsf{A}_2$, $\mathsf{D}_4$ 
an $\mathsf{E}_8$ --- is roughly speaking as follows:  
First one proves local extremality for all $c$ bigger than an
\emph{explicit} $c_0$ (as small as possible, but certainly not $0$!), 
and then, if $c_0$ is small enough, one can use self-duality together 
with the Poisson summation formula to switch from  "big $c$'' to "small
$c$" (see \cite{MR2889159} for details). In our situation here, there are
two difficulties in applying this strategy. First, as explained in
\cite{MR3289409}, there is no good notion of duality, let alone
self-duality, and the Poisson summation formula for general periodic
sets. This first obstruction seems unavoidable, and incidentally one
does not expect universal local optimality of $\mathsf{D}_n^{+}$ for general
$n\geq 8$.  But fortunately 
the $2$-periodic set $\mathsf{D}_n^{+}$ (with $n$ odd) 
is precisely one instance of a non-lattice configuration
for which a formal self-duality holds together with a Poisson formula
(see~\cite{PhysRev09}). 
So it is not hopeless to overcome this first obstruction in this
particular case. 
The second impediment, not theoretical in nature but really critical
in practice, is the need for an \emph{explicit} threshold $c_0$. To
this end, one needs an effective version of Lemma \ref{cusp}, \ie
effective bounds for the coefficients of the cusp form involved, in
the spirit of \cite{MR2854563} for instance. 
But those seem to be quite difficult in our case, 
given that the cusp form $\theta_{\Lambda,P}(\tau)$ has half-integral weight.
Here, further research appears to be necessary.

\section*{Acknowledgments}

Both authors were supported 
by the Erwin-Schr\"odinger-Institute (ESI) during a stay in fall 2014 for
the program on Minimal Energy Point Sets, Lattices and Designs.
The second author gratefully acknowledges support by 
DFG grant SCHU 1503/7-1. 
The authors like to thank Jeremy Rouse,
Frieder Ladisch and Robert Sch\"uler for several valuable
remarks.



\providecommand{\bysame}{\leavevmode\hbox to3em{\hrulefill}\thinspace}
\providecommand{\MR}{\relax\ifhmode\unskip\space\fi MR }
\providecommand{\MRhref}[2]{%
	\href{http://www.ams.org/mathscinet-getitem?mr=#1}{#2}
}
\providecommand{\href}[2]{#2}

\end{document}